\documentclass[12pt]{amsart}
\usepackage[usenames]{color}
\usepackage{enumerate}
\usepackage{mathdots,bm,url,mathscinet}
\usepackage{amsfonts,amssymb,a4wide}
\usepackage[all]{xy}
\usepackage[stable]{footmisc}

\author{Erez Lapid}
\address{Department of Mathematics, Weizmann Institute of Science, Rehovot 76100 Israel}
\email{erez.m.lapid@gmail.com}
\author{Zhengyu Mao}
\address{Department of Mathematics and Computer Science, Rutgers University, Newark, NJ 07102, USA}
\email{zmao@rutgers.edu}
\title[Whittaker--Fourier coefficients on unitary groups]
{On Whittaker--Fourier coefficients of automorphic forms on unitary groups: reduction to a local identity}
\date{\today}
\thanks{Authors partially supported by U.S.--Israel Binational Science Foundation Grant \# 057/2008}
\thanks{Second named author partially supported by NSF grant DMS 1400063}
\dedicatory{In honor of Jim Cogdell for his 60th birthday}

\newcommand{\A}{\mathbb{A}}                            
\newcommand{\C}{\mathbb{C}}                            
\newcommand{\Q}{\mathbb{Q}}                            
\newcommand{\R}{\mathbb{R}}                            
\newcommand{\AF}{\mathcal{A}}                          
\newcommand{\bs}{\backslash}
\newcommand{\vol}{\operatorname{vol}}

\newcommand{\OO}{\mathcal{O}}                         
\newcommand{\uno}{\operatorname{U}^+_{2n+1}}
\newcommand{\une}{\operatorname{U}^-_{2n}}
\newcommand{\Mat}{\operatorname{Mat}}
\newcommand{\Ad}{\operatorname{Ad}}
\newcommand{\GL}{\operatorname{GL}}
\newcommand{\GLE}{\operatorname{Res}_{E/F}\operatorname{GL}}
\newcommand{\Sp}{\operatorname{Sp}}
\newcommand{\U}{\operatorname{U}}
\newcommand{\Ind}{\operatorname{Ind}}                   
\newcommand{\Levi}{M}
\newcommand{\GLnn}{{\mathbb M}}
\newcommand{\GLn}{{{\mathbb M}'}}
\newcommand{\bigune}{\operatorname{U}^-_{4n}}
\newcommand{\biguno}{\operatorname{U}^+_{4n+2}}
\renewcommand{\c}[1]{\mathfrak{c}(#1)}
\newcommand{\chiq}{\eta_{E/F}}
\newcommand{\chiqe}{\Upsilon}
\newcommand{\GG}{\operatorname{GG}}
\newcommand{\stint}{\int^{st}}                         
\newcommand{\eisen}{\mathcal{E}}                       
\newcommand{\reseisen}{\eisen_{-k}}                    
\newcommand{\resm}{m_{-k}}                             
\newcommand{\desc}{\mathcal{D}}
\newcommand{\des}{\mathcal{D}_\psi}                     
\newcommand{\desinv}{\mathcal{D}_{\psi^{-1}}}

\newcommand{\whitform}{A^{\psi,\chiqe}}                              
\newcommand{\whitformo}{A^{\psi}}                              
\newcommand{\whitformd}{A^{\psi^{-1},\chiqe^{-1}}}
\newcommand{\whitformod}{A^{\psi^{-1}}}
\newcommand{\Mint}{A_e^{\psi,\chiqe}}                                
\newcommand{\Mintd}{A_e^{\psi^{-1},\chiqe^{-1}}}
\newcommand{\wgt}[1]{\nu(#1)}
\renewcommand{\d}[1]{#1^{\vee}}
\newcommand{\alt}[1]{#1^{\wedge}}
\newcommand{\Irr}{\operatorname{Irr}}                  
\newcommand{\ut}{\operatorname{ut}}
\newcommand{\rest}{\big|}                              
\newcommand{\Cusp}{\operatorname{Cusp}}                
\newcommand{\Ucusp}{\operatorname{Ucusp}}
\newcommand{\gen}{\operatorname{gen}}                  
\newcommand{\As}{\operatorname{As}}
\newcommand{\levi}{\varrho}                            
\newcommand{\toU}{\ell}
\newcommand{\toUbar}{\overline{\ell}}
\newcommand{\startran}[1]{\breve{#1}}             
\newcommand{\swrz}{\mathcal{S}}
\newcommand{\modulus}{\delta}
\newcommand{\diag}{\operatorname{diag}}
\newcommand{\whit}{\mathcal{W}}                          
\newcommand{\Whit}{\mathbb{W}}                           
\newcommand{\WhitM}{\Whit^{\psi_{N_\GLnn}}}                  
\newcommand{\WhitMd}{\Whit^{\psi_{N_\GLnn}^{-1}}}
\newcommand{\WhitML}{\Whit^{\psi_{N_\Levi}}}                  
\newcommand{\WhitMLF}{\Whit^{\psi_{N_\Levi}^*}}                  
\newcommand{\WhitMLd}{\Whit^{\psi_{N_\Levi}^{-1}}}
\newcommand{\WhitGd}{\Whit^{\psi_{N'}^{-1}}}
\newcommand{\antisym}{\mathfrak{a}}
\newcommand{\psiweil}{\psi_{\circ}}
\newcommand{\Nc}{{\mathcal B}}   
\newcommand{\few}{\hat{w}}              
\newcommand{\one}{\epsilon}        
\newcommand{\tr}{\operatorname{tr}}                        
\newcommand{\Hei}[1]{{#1}_{{\mathcal H}}}
\newcommand{\sprod}[2]{\left\langle#1,#2\right\rangle}
\newcommand{\abs}[1]{\left|{#1}\right|}
\newcommand{\sm}[4]{\left(\begin{smallmatrix}{#1}&{#2}\\{#3}&{#4}\end{smallmatrix}\right)}
\newcommand{\weil}{\omega}
\newcommand{\wev}{\weil_{\psi_{N_\GLnn}}}
\newcommand{\wevinv}{\weil_{\psi_{N_\GLnn}^{-1}}}
\newcommand{\Ne}{\mathfrak{E}}
\newcommand{\FJ}{\operatorname{FJ}_{\psi_{N_\GLnn}, \chiqe}}   
\newcommand{\rr}{\tau}
\newcommand{\alltri}{{\mathcal A}}
\newcommand{\unitE}{E^u}

\newtheorem{theorem}{Theorem}[section]
\newtheorem{lemma}[theorem]{Lemma}
\newtheorem{proposition}[theorem]{Proposition}
\newtheorem{remark}[theorem]{Remark}
\newtheorem{conjecture}[theorem]{Conjecture}

\newtheorem{claim}[theorem]{``Claim''}
\newtheorem{assumption}[theorem]{Working Assumption}

\numberwithin{equation}{section}

\begin{document}

\begin{abstract}
We study Whittaker--Fourier coefficients of automorphic forms on a quasi-split unitary group.
We reduce the analogue of the Ichino--Ikeda conjectures to a conjectural local statement
using the descent method of Ginzburg--Rallis--Soudry. 
\end{abstract}

\maketitle

\setcounter{tocdepth}{1}
\tableofcontents
\section{Introduction}

In \cite{MR3267120} we studied the Whittaker--Fourier coefficients of cusp forms on adelic quotients
of quasi-split groups over number fields and formulated a conjecture relating them to the Petersson inner product.
In the case of the metaplectic groups $\widetilde{\Sp}_n$ we further reduced the global conjecture to a local conjecture
in \cite{1401.0198}. In the $p$-adic case we proved the local conjecture in \cite{1404.2905}.

In this note, we turn our attention to the case of (quasi-split) unitary groups. Let us recall the conjecture of \cite{MR3267120} in this case.
Let $E/F$ be a quadratic extension of number fields and $\A$ the ring of adeles of $F$.
Let $\U_n$ be a quasi-split unitary group and $N'$ a maximal unipotent subgroup of $\U_n$.
Fix a non-degenerate character $\psi_{N'}$ on $N'(\A)$, trivial on $N'(F)$.
For a cusp form $\varphi$ of $\U_n(F)\bs\U_n(\A)$ we consider the Whittaker--Fourier coefficient
\[
\whit(\varphi)=\whit^{\psi_{N'}}(\varphi):=
(\vol(N'(F)\bs N'(\A)))^{-1}\int_{N'(F)\bs N'(\A)} \varphi(u)\psi_{N'}(u)^{-1}\ du.
\]
If $\varphi^\vee$ is another cusp form on $\U_n(F)\bs\U_n(\A)$ we also set
\begin{equation}\label{eq: innerdef}
(\varphi, \varphi^\vee)_{\U_n}=(\vol(\U_n(F)\bs\U_n(\A)))^{-1}\int_{\U_n(F)\bs\U_n(\A)} \varphi(g) \varphi^\vee(g)\ dg.
\end{equation}

Given a finite set of places $S$ of $F$ we defined in \cite{MR3267120} a regularized integral
\[
\stint_{N'(F_S)}f(u)\ du
\]
for a suitable class of smooth functions $f$ on $N'(F_S)$. If $S$ consists only of non-archimedean places then
\[
\stint_{N'(F_S)}f(u)\ du=\int_{N_1'}f(u)\ du
\]
for any sufficiently large compact open subgroup $N_1'$ of $N'(F_S)$.
(In the archimedean case an ad-hoc definition is given.)

Let $\sigma$ be an irreducible generic cuspidal representation of $\U_n(\A)$.
By \cite[Ch. 11]{MR2848523} the weak lift $\pi$ of $\sigma$ to $\GL_n(\A_E)$ (which exists by \cite{MR2767514})
is an isobaric sum $\pi_1\boxplus\dots\boxplus\pi_k$ of pairwise inequivalent irreducible cuspidal representations $\pi_i$ of
$\GL_{n_i}(\A_E)$, $i=1,\dots,k$ (with $n_1+\dots+n_k=n$) such that $L^S(s,\pi_i,\As^{(-1)^{n-1}})$ has a pole (necessarily simple) at $s=1$ for all $i$.
Here $L^S(s,\pi_i,\As^{\pm})$ are the (partial) Asai $L$-functions of $\pi_i$.
Our convention is that $L(s,\pi,\As^+)$ is the Asai $L$-function $L(s,\pi,\As)$ as defined in \cite[\S2.3]{MR2848523}),
while $L(s,\pi,\As^-)=L(s,\pi\otimes\chiqe,\As^+)$ where $\chiqe$ is any Hecke character of $\A_E^{\times}$ whose restriction to
$\A_F^{\times}$ is the quadratic character $\chiq$ associated to the extension $E/F$.

\begin{conjecture}(\cite[Conjecture 1.2,5.1]{MR3267120}) \label{conj: global}
Assume that $\sigma$ weakly lifts to $\pi$ as above.
Then for any $ \varphi\in\sigma$ and $\d{ \varphi}\in\d{\sigma}$ and for any sufficiently large finite set $S$ of places of $F$ we have
\begin{multline} \label{eq: globalidentity}
 \whit^{\psi_{N'}}(\varphi) \whit^{\psi_{N'}^{-1}}(\d{\varphi})=2^{1-k}\frac{\prod_{j=1}^nL^S(j,\chiq^j)}{L^S(1,\pi,\As^{(-1)^n})}\times\\
(\vol(N'(\OO_S)\bs N'(F_S)))^{-1}\stint_{N'(F_S)}(\sigma(u)\varphi,\d{\varphi})_{\U_n}\psi_{N'}(u)^{-1}\ du.
\end{multline}
Here $\OO_S$ is the ring of $S$-integers of $F$.
\end{conjecture}

We will follow the treatment of \cite{1401.0198} to reduce the above conjecture to a conjectural local identity.
We will also give a heuristic argument for the conjectural local identity for the cases $n=2$ or $3$.
The reduction to a local identity is based on the work of Ginzburg--Rallis--Soudry on automorphic descent.
The descent construction for cuspidal representations of $\U_n$ depends on the parity of $n$,
so we have to treat the cases $n$ even and $n$ odd separately.
At the moment, the descent theory is more thoroughly developed in the case of metaplectic groups than in the case of unitary groups
and there are some expected properties of the descent which are not yet established in the latter case.
Although it is likely that the same methods work, we will not concern ourselves with bridging these gaps here.
Instead, we will take for granted the expected properties of the descent for unitary groups.
Thus, our results are conditional.

Finally, we mention that the putative local identity is expected to be equivalent to the formal degree conjecture
of Hiraga--Ichino--Ikeda \cite{MR2350057} in the case of generic square-integrable representations. (See \cite{1404.2909} for the case
of odd orthogonal and metaplectic groups where the formal degree conjecture is established using this approach.)

\nocite{MR3267124}

\subsection{Acknowledgement}
It is a pleasure to dedicate this paper to Jim Cogdell.
Jim has been an inspirational figure in automorphic forms.
On a personal level, Jim has always been very supportive.
The second named author is especially grateful to Jim for being his postdoctoral mentor in the early 1990s.

We also thank the anonymous referee for carefully reading the paper.


\section{Notation and preliminaries}

Let $F$ be a local field of characteristic zero.
By abuse of notation we will use the same letter for an algebraic group over $F$ and its group of $F$-points.
We denote by $\Irr Q$ the set of (equivalence classes of) smooth complex irreducible representations of the group of $F$-points of
an algebraic group $Q$ over $F$. We also write $\modulus_Q$ for the modulus function of $Q$.
If $Q$ is quasi-split and $\psi_{N_Q}$ is a non-degenerate character of a maximal unipotent subgroup $N_Q$,
we denote by $\Irr_{\gen,\psi_{N_Q}} Q$ the subset of representations that are generic with respect to $\psi_{N_Q}$.
We suppress $\psi_{N_Q}$ from the notation if it is clear from the context or is irrelevant.
When $\pi\in \Irr_{\gen, \psi_{N_Q}} Q$, let $\Whit^{\psi_{N_Q}}(\pi)$ be the Whittaker model of $\pi$.

Let $E$ be a quadratic \'etale algebra over $F$ and $\c{\cdot}$ the nontrivial $F$-automorphism of $E$.
We denote by $\abs{\cdot}$ the normalized absolute value of $E$.
Let $\chiq$ be the corresponding quadratic character of $F^\times$. (It is trivial if $E/F$ is split.)

Let $\Mat_{l,m}$ be the space of $l\times m$ matrices. Let $\one_{i,j}^m$ denote the $m\times m$ matrix $x$ such that $x_{i,j}=1$ and $x_{k,l}=0$ for all other entries.
Let $w_l$ be the $l\times l$-matrix $\sum_{i=1}^l \one_{i,l+1-i}^l$.
Let $g\mapsto g^*$ be the outer automorphism of $\GL_l(E)$ given by $g^*=w_l^{-1}\,\c{\,^tg^{-1} }w_l$.

Let $J_m^\pm=\sm{}{w_m}{\pm w_m}{}$ and $\U_{2m}^\pm=\{g\in\GL_{2m}(E):\, \c{\,^tg}J_{m}^\pm g=J_{m}^\pm\}$,
the quasi-split unitary group, acting on the $2m$-dimensional hermitian/skew-hermitian space with standard basis $e_1,\ldots,e_{m},e_{-m},\ldots,e_{-1}$.
If $E/F$ is split, $\U_{2m}(F)\simeq \GL_{2m}(F)$.
Let $P=\Levi\ltimes U$ be the Siegel parabolic subgroup of $\U_{2m}^\pm$ with Levi part $\Levi=\levi(\GLnn)$
where $\GLnn=\GLE_{m}$ and $\levi: h\mapsto \diag(h,h^*)$.
Let $K=K_{\GL_{2m}(E)}\cap\U_{2m}^\pm$ where $K_{\GL_{2m}(E)}$ is the standard maximal compact subgroup of $\GL_{2m}(E)$.
Thus $K$ is a maximal compact subgroup of $\U_{2m}^\pm$. 
Using the Iwasawa decomposition we extend the character $\levi(g)\mapsto\abs{\det g}$, $g\in\GLnn$ to a right
$K$ left $U$ invariant function $\wgt{g}$ on $U_{2m}^\pm$.

For the rest of the paper we will consider either $G=\U_{4n}^-$ or $\U_{4n+2}^+$ (so that $m=2n$ or $2n+1$ respectively).
For any $f\in C^\infty(G)$ and $s\in\C$ define $f_s(g)=f(g)\wgt{g}^s$, $g\in G$.

Let $N_\Levi$ be the standard maximal unipotent subgroup of $\Levi$ and $\psi_{N_\Levi}$ a non-degenerate
character of $N_\Levi$.
Let $\pi$ be an irreducible generic representation of $\Levi$ with Whittaker model $\WhitML(\pi)$.
Let $\Ind(\WhitML(\pi))$ be the space of $G$-smooth left $U$-invariant functions $W:G\rightarrow\C$ such that
for all $g\in G$, the function $m\mapsto\modulus_P(m)^{-\frac12}W(mg)$ on $\Levi$ belongs to $\WhitML(\pi)$.
For any $s\in\C$ we have a representation $\Ind(\WhitML(\pi),s)$ on the space $\Ind(\WhitML(\pi))$ given by
$(I(s,g)W)_s(x)=W_s(xg)$, $x,g\in G$.

Let $w_U=\sm{}{I_m}{\pm I_m}{}\in G$ (where the sign $\pm$ is $(-1)^{m+1}$).
Define the intertwining operator $M(\pi,s)=M(s):\Ind(\WhitML(\pi),s)\rightarrow\Ind(\WhitMLF(\d\pi),-s)$ by
(the analytic continuation of)
\begin{equation} \label{eq: defM}
M(s)W(g)=\wgt{g}^s\int_U W_s( w_U ug)\,du
\end{equation}
where $\psi_{N_\Levi}^*(u)=\psi_{N_\Levi}(w_Uuw_U^{-1})$ and $\d\pi$ is the contragredient of $\pi$.

In the case where $F$ is $p$-adic with $p$ odd, $E/F$ and $\pi$ are unramified, and there exist (necessarily unique)
$K$-fixed elements $W^\circ\in\Ind(\WhitML(\pi))$ and $W'^\circ\in\Ind(\WhitMLF(\d\pi))$ such that $W^\circ(e)=W'^\circ(e)=1$
then we have (assuming $\vol(U\cap K)=1$)
\begin{equation} \label{eq: unramwhit}
M(s)W^\circ=\frac{L(2s,\pi,\As^+)}{L(2s+1,\pi,\As^+)}W'^\circ.
\end{equation}

The following result is an analogue of \cite[Proposition 4.1]{1401.0198}.
The proof is almost identical, and will be omitted.
\begin{proposition} \label{prop: M1/2}
Suppose that $\pi\in\Irr_{\gen}\GL_m(E)$ is such that $\d\pi\cong \c{\pi}$. Then $M(\pi,s)$ is holomorphic at $s=\frac12$.
\end{proposition}

\section{Representations of unitary type}

\subsection{Global setting}
Let  $F$ be a number field and $E$ a quadratic extension of $F$. Let $\A$ (resp., $\A_E$) be the ring of adeles of $F$ (resp., $E$).
Denote by $\Cusp\GL_m(E)$ the set of irreducible cuspidal automorphic representations of $\GL_m(\A_E)$ whose central character is
trivial on the positive reals (where $\R\hookrightarrow\A_{\Q}\hookrightarrow\A_E$).
We say that $\pi\in\Cusp\GL_m(E)$ is of \emph{unitary type} if
\[
\int_{\GL_m(F)\bs \GL_m(\A)^1} \varphi(h)\ dh\ne0
\]
(where $\GL_m(\A)^1=\{g\in\GL_m(\A):\abs{\det g}=1\}$)
for some $\varphi$ in the space of $\pi$. In particular, this implies that the central character of $\pi$ is trivial on $\A_F^\times$.

The following characterization is due to Flicker--Zinoviev.
\begin{proposition} (\cite{MR1344660}) \label{P: flicker}
Let $\pi\in\Cusp\GL_m(E)$.
Then $\pi$ is of unitary type if and only if $L^S(s,\pi,\As^+)$ has a pole at $s=1$.
\end{proposition}

\subsection{Local setting}

We say that $\pi\in \Irr \GL_m(E)$ is of unitary type if it has a nontrivial $\GL_m(F)$-invariant linear form.
In particular, the central character of $\pi$ is trivial on $F^\times$.
We write $\Irr_{\ut}\GL_m(E)$ for the set of irreducible representations of unitary type.
Clearly, if $\pi$ is of unitary type in the global setting then all its local components $\pi_v$ are of unitary type.

We recall some results on local representations of unitary type, due to Aizenbud--Gourevitch, Flicker, Jacquet--Shalika and Kable.
\begin{lemma}
\begin{enumerate}
\item \label{one} (\cite{MR1111204} -- $p$-adic case; \cite{MR2553879} -- archimedean case)
Suppose that $\pi\in\Irr_{\ut}\GL_m(E)$. Then the space of $\GL_m(F)$-invariant linear forms on $\pi$ is one-dimensional. Moreover $\d\pi\cong\c{\pi}$.
\item \label{pole} (\cite{MR2075482} -- inert case; \cite{MR618323} -- split case)
Suppose that $\pi\in\Irr\GL_m(E)$ is square integrable. Then $\pi\in\Irr_{\ut}\GL_m(E)$ if and only if $L(s,\pi,\As^+)$ has a pole at $s=0$.
(The local $L$-function is the one defined by Shahidi -- cf.~\cite{MR1168488}.)
\item \label{ind} (same proof as \cite[Lemma~3.5]{1401.0198})
Suppose that $\pi_i\in\Irr_{\ut}\GL_{m_i}(E)$, $i=1,2$ and the parabolic induction $\pi_1\times\pi_2$ is irreducible.
Then $\pi_1\times\pi_2\in\Irr_{\ut}\GL_{m_1+m+2}(E)$.
\end{enumerate}
\end{lemma}

For completeness we also recall the following classification theorem, due to Matringe, of the set $\Irr_{\gen,\ut}\GL_m(E)$ of generic representations of unitary type.

\begin{proposition} (\cite[Theorem 5.2]{MR2755483})
Assume that $F$ is $p$-adic and $E/F$ is inert.
Then the set $\Irr_{\gen,\ut}\GL_m(E)$ consists of the irreducible representations of the form
\[
\pi=\c{\sigma_1}\times\d\sigma_1\times\dots\times\c{\sigma_k}\times\d\sigma_k\times
\tau_1\times\dots\times\tau_l
\]
where $\sigma_1,\dots,\sigma_k$ are essentially square-integrable,
$\tau_1,\dots,\tau_l$ are square-integrable of unitary type (i.e., $L(0,\tau_i,\As^+)=\infty$ for all $i$).
\end{proposition}
(For the archimedean analogue, see the work of Kemarsky \cite{1408.4299}.)

When $E/F$ is split, the generic representations of unitary type are of the form $\pi_1\otimes\pi_1^\vee$
where $\pi_1$ is a generic representation of $\GL_n(F)$. The latter are classified in terms of
essentially square integrable representations.

\subsection{Linear form on induced representation} \label{sec: inflate}
Now let $\pi\in \Irr_{\ut}\GL_m(E)$ and $\ell$ a nontrivial $\GL_m(F)$-invariant linear form on the space of $\pi$.
We also consider $\pi$ as a representation of $\Levi\subset G$ via $\levi$. Let $\Pi=\Ind_P^G(\pi\nu^\frac12)$.

Consider first the case $G=\U^-_{2m}$.
The group $H=\GL_{2m}(F)\cap G$ is the symplectic group $\Sp_{m}$ of rank $m$.
We define a linear form $L$ on $\Pi$  by
\begin{equation}\label{eq: inflate}
L(\phi)=\int_{(P\cap H)\bs H} \ell(\phi(g))\ dg.
\end{equation}
This is well defined since $\modulus_{P\cap H}=\nu^{\frac12}\modulus_P^{\frac12}\rest_{P\cap H}$.
Clearly $L$ is an $H$-invariant linear form on $\Pi$.
Moreover $L(\Pi(\levi(z))\phi)=\omega_\pi(z)L(\phi)$ where $z$ is in the center of $\GLnn$ and $\omega_\pi$ is the central character of $\pi$.
Let $H'\subset G$ be the group
\[
H'=\{h\in G:h=\lambda(h)\c{h}\text{ for some scalar }\lambda(h)\}\subset G.
\]
Note that the character $\lambda$ of $H'$ takes values in $\unitE$, the group of norm $1$ elements in $E$.
Then $L(\pi(h)\phi)=\omega'_\pi(\lambda(h))L(\phi)$ for $h\in H'$ where $\omega'_\pi$ is the character of $\unitE$ given by
$\omega_\pi(x)=\omega'_\pi(x \c{x}^{-1})$ (where we recall that $\omega_\pi$ is trivial on $F^{\times}$).

Next consider the case  $G=\U^+_{2m}$.
Fix a non-zero element $\rr\in E$ such that $\c{\rr}=-\rr$ and let ${\bf r}=\diag(\rr I_m,I_m)$.
Notice that the map $g\mapsto {\bf r}^{-1}g{\bf r}$ is an isomorphism between $\U^+_{2m}$ and $\U^-_{2m}$.
Thus, if we let $H={\bf r}^{-1}\GL_{2m}(F){\bf r}\cap G$ in the expression \eqref{eq: inflate} above then the linear form on $\Pi$ defined by \eqref{eq: inflate}
satisfies $L(\pi(h)\phi)=\omega'_\pi(\lambda(h))L(\phi)$ for $h\in H'$, where $H'$ is the subgroup of $h\in G$ such that $\diag(I_m,-I_m)h=\lambda(h) \c{h}\diag(I_m,-I_m)$.

We observe that in both cases we have:
\begin{equation}\label{eq: inflate2}
L(\phi)=\int_{(P\cap H')\bs H'} \ell(\phi(g))\omega_{\pi}'(\lambda(g)^{-1})\ dg.
\end{equation}

\part{The case $\une$}

\section{Fourier--Jacobi coefficients and descent}

Let $G=\bigune$. Recall in this case $\GLnn=\GLE_{2n}$.
Let $G'\subset\bigune$ be the subgroup consisting of elements fixing $e_1,\ldots,e_n$ and $e_{-1},\ldots,e_{-n}$.
Thus $G'\simeq\U_{2n}^-$. Let $K'=G'\cap K$.
Let $\GLn=\GLE_n$ and let $M'$ be its image in $G'$ under $\levi'$ where $\levi'(g)=\diag(1_n,g,g^*,1_n)$.

\subsection{Characters} \label{sec: characters}
Let $N_\GLnn$ be the standard maximal unipotent subgroup of $\GLnn$.
We fix a non-degenerate character $\psi_{N_\GLnn}$ of $N_\GLnn$ such that
the character $\psiweil(x):=\psi_{N_\GLnn}(I_{2n}+x\one_{n,n+1}^{2n})$, $x\in E$ satisfies $\psiweil(x)=\psiweil(\c{x})$.
As in \cite{1401.0198}, the statements in the sequel will not depend on the choice of $\psi_{N_\GLnn}$ (cf.~[ibid., Remark~6.4]).

The character $\psi_{N_\GLnn}$ determines additional characters on various unipotent groups as follows:
\begin{itemize}
\item $N_\Levi:=\levi(N_\GLnn)$; $\psi_{N_\Levi}(\levi(u))=\psi_{N_\GLnn}(u)$, $u\in N_\GLnn$.

\item $N'_\GLn$ is the standard maximal unipotent subgroup of $\GL_n$; $\psi_{N'_\GLn}(u')=\psi_{N_\GLnn}(\diag(u',1_n))$.

\item $\psi_{N'_{\Levi'}}$ is the non-degenerate character of $N'_{\Levi'}:=\levi'(N_\GLn)$ such that $\psi_{N'_{\Levi'}}(\levi'(u))=\psi_{N'_{\GLn}}(u)$.

\item $U'$ be the Siegel unipotent subgroup of $G'$; $\psi_{U'}$ is the character on $U'$ given by $\psi_{U'}(u)=\psiweil(\frac12 u_{2n,2n+1})^{-1}$.

\item $N'=N'_{\Levi'}\ltimes U'$; $\psi_{N'}(uv)=\psi_{N'_{\Levi'}}(u)\psi_{U'}(v)$ with $u\in N'_{\Levi'}$ and $v\in U'$.

\item $N=N_\Levi\ltimes U$ where $U$ is the Siegel unipotent of $G$; $\psi_N$ is the \emph{degenerate} character on $N$ given by
$\psi_N(uv)=\psi_{N_\Levi}(u)$ for any $u\in N_\Levi$ and $v\in U$.

\item $V$ is the unipotent radical in $G$ of the standard parabolic subgroup with Levi $\GL_1^n\times G'$.
An element in $V$ can be written as $vu$ where $u$ fixes $e_1,\ldots e_n$, $v$ fixes $e_{n+1},e_{n+2},\ldots,e_{-1-n}$ and we set
$\psi_V(vu)=\psi_{N_\Levi}( w_U v w_U^{-1})$.
\end{itemize}

For convenience, we will fix a non-trivial character $\psi$ of $E$ satisfying $\psi(x)=\psi(\c{x})$ and set
\[
\psi_{N_\GLnn}(u)=\psi(u_{1,2}+\dots+u_{2n-1,2n}).
\]
Thus $\psiweil=\psi$. Note that this choice is different from the conventions of \cite{MR2848523}.
With this choice of $\psi_{N_\GLnn}$ we have
\begin{eqnarray*}
\psi_{N'_\GLn}(u')&=&\psi(u'_{1,2}+\dots+u'_{n-1,n})\\
\psi_V(v)&=&\psi(v_{1,2}+\dots+v_{n-1,n})^{-1}, \ \ v\in V\cap\Levi.
\end{eqnarray*}

\subsection{Weil representation}

Let $V_0\subset V$ be the unipotent radical of the standard parabolic subgroup of $G$ with Levi $\GL_1^{n-1}\times\U^-_{2n+2}$.
Then the map
\[
v\mapsto\Hei{v}:=((v_{n,n+j})_{j=1,\dots,2n},\frac14\tr_{E/F}(v_{n,3n+1}))
\]
gives an isomorphism from $V/V_0$ to a Heisenberg group $W\oplus F$ with
\[
(w_1,t_1)\cdot (w_2,t_2)=(w_1+w_2,t_1+t_2+\frac12\tr_{E/F}\sprod{w_1}{w_2})
\]
where $\sprod{\cdot}{\cdot}$ is the  $F$-bilinear symplectic form on $W\cong E^{2n}$
\[
\sprod{(x_1,\dots,x_{2n})}{(y_1,\dots,y_{2n})}=\sum_{i=1}^n  x_i\c{y_{2n+1-i}}-\sum_{i=1}^ny_i\c{x_{2n+1-i}}.
\]
Let $\widetilde{\Sp(W)}$ be the metaplectic double cover of $\Sp(W)$.
We denote by $\weil_{\psiweil}$ the Weil representation of $\widetilde{\Sp(W)}\ltimes V/V_0$ determined by the additive character $\psiweil$,
realized on $\swrz(E^n)$ where we identify $E^n$ with the first $n$-coordinates of $W$.

It is clear that $G'\subset \Sp(W)$.
Moreover it is known that the metaplectic cover of $\Sp(W)$ splits (non-canonically) over $G'$.
We choose the splitting as in \cite{MR2848523}. It depends on a choice
of a character $\chiqe$ of $E^{\times}$ such that $\chiqe\rest_{F^{\times}}=\chiq$.
We denote the restriction of the Weil representation to $G'$ (with respect to that splitting) by $\weil_{\psiweil}^{\chiqe}$.

\begin{remark}
Our $\weil_{\psiweil}$ corresponds to the definition given in \cite[(1.5)]{MR2848523} with the character $\psi(x)$ there replaced by $\psiweil(\frac12 x)$.
We find it more convenient to use this convention.
\end{remark}

We extend $\weil_{\psiweil}^{\chiqe}$ to a representation $\wev^{\chiqe}$ of $V\rtimes G'$ by setting
\begin{equation}\label{eq: weilext}
\wev^{\chiqe}(v g)\Phi=\psi_V(v)\weil_{\psiweil}^{\chiqe}(\Hei{v})(\weil_{\psiweil}^{\chiqe}( g)\Phi),\,\,v\in V,\,\,g\in G'.
\end{equation}

\subsection{Fourier--Jacobi coefficient}
Suppose now that $E$ is a quadratic extension of a number field $F$.
Let $\chiq$ be the associated quadratic character on $\A_F^{\times}$ and $\chiqe$ an extension of it to $E^{\times}\bs \A_E^{\times}$.
All the previous notation has an obvious meaning in the global context. 

For $\Phi\in\swrz(\A_E^{n})$ define the theta function
\[
\Theta^\Phi_{\psi_{N_\GLnn},\chiqe}(v  g)=\sum_{\xi\in E^{n}}\wev^{\chiqe}(v g)\Phi(\xi),\ \ v\in V(\A), g\in G'(\A).
\]

For any automorphic form $\varphi$ on $G(\A)$ and $\Phi\in\swrz(\A_E^{n})$, let
$\FJ(\varphi,\Phi)$ be the Fourier--Jacobi coefficient (a function on $G'(F)\bs G'(\A)$)
\begin{equation} \label{def: FJ}
\FJ(\varphi,\Phi)(g)=\int_{V(F)\bs V(\A)}\varphi(vg)\Theta^\Phi_{\psi_{N_\GLnn}^{-1},\chiqe^{-1}}(vg)\,dv,\ \ g\in G'(\A).
\end{equation}

\subsection{Descent map}\label{unesec: desc}

We let $\Ucusp \GLnn$ be the set of  automorphic representations $\pi$ of the form $\pi_1\times\ldots\times\pi_k$, where
$\pi_i\in\Cusp\GL_{n_i}(E)$, $i=1,\dots,k$ are pairwise inequivalent and of unitary type (with $n_1+\dots+n_k=2n)$.
In particular $\d{\pi_i}\cong \c{\pi_i}$ and the central character of $\pi$ is trivial on $\A^{\times}_F$.
Note also that if $\pi\in \Ucusp \GLnn$, then $\pi\otimes\chiqe\not\in \Ucusp\GLnn$.

The automorphic representation $\pi$ is realized on the space of Eisenstein series induced from $\pi_1\otimes\dots\otimes\pi_k$.

For $\pi\in \Ucusp\GLnn$, we view it as a representation of $\Levi(\A)$ via $\levi$.
Let $\AF(\pi)$ be the space of functions $\varphi:\Levi(F)U(\A)\bs G(\A)\rightarrow\C$
such that $m\mapsto\modulus_P(m)^{-\frac12}\varphi(mg)$, $m\in\Levi(\A)$ belongs to the space of $\pi$ for all $g\in G(\A)$.
(Here $\modulus_P$ is the modulus function on the Siegel parabolic $P=M\ltimes U$).
One can associate a space of Eisenstein series $\{\eisen(\varphi,s)\}$ (where $\phi\in\AF(\pi)$) on $G$, where
$\eisen(\varphi,s)$ has a pole of order $k$ at $s=\frac12$ (see \cite[Theorem~2.1]{MR2848523}).
Set $\reseisen\varphi=\lim_{s\rightarrow\frac12}(s-\frac12)^k\eisen(\varphi,s)$.
The descent of $\pi$ is the space $\desc_{\psi}^{\chiqe}(\pi)$ generated by $\FJ(\reseisen\varphi,\Phi)$, $\varphi\in\AF(\pi)$, $\Phi\in\swrz(\A_E^n)$.

Note that the descent map depends on the choice of $\chiqe$.
By \cite[Theorem~9.7]{MR2848523}, $\desc_{\psi }^{\chiqe}(\pi)$ is a nontrivial cuspidal automorphic representation of $G'$.
It is known that $\desc_{\psi }^{\chiqe}(\pi)$ is multiplicity free (\cite[Theorem~3.1]{MR2848523}).
We expect it to be irreducible. This would follow from the analogue of \cite[Theorem~5.3]{MR1983781} in the unitary group case.
It is likely that the methods of [loc.~cit.] extend to the case at hand.
However, since this is beyond the scope of the current paper we will simply make it a working assumption.

\begin{assumption}\label{uneassume}
When $\pi \in\Ucusp\GLnn$,  $\desc_{\psi }^{\chiqe}(\pi)$ is irreducible.
\end{assumption}
It then follows from \cite[Theorem~11.2]{MR2848523} that
\begin{proposition}\label{prop: ueven}
With the above assumption, $\pi\mapsto \sigma=\desc_{\psi }^{\chiqe}(\pi)$ defines a bijection between $\Ucusp\GLnn$ and
$\Cusp_{\psi_{N'}} G'$, the set of $\psi_{N'}$-generic irreducible cuspidal automorphic representations of $G'$.
Moreover $\d\sigma$ weakly lifts to $\c{\pi}\otimes\chiqe$.
\end{proposition}

\section{Reduction to a local conjecture} \label{sec: redeven}

\subsection{Explicit local descent} \label{sec: local conjecture}


Let $\pi\in \Irr_{\gen}\Levi$. For any $W\in\Ind(\WhitML(\pi))$ and $\Phi\in \swrz(E^n)$ define the following function on $G'$:
\begin{equation}\label{eq: defwhitform}
\whitform(W,\Phi, g,s)=\int_{V_\gamma\bs V} W_s(\gamma v g)\wevinv^{\chiqe^{-1}}(v g)\Phi(\xi_n)\,dv,\,\,\,g\in G'
\end{equation}
where $\gamma=\left(\begin{smallmatrix}&I_n&&\\&&&I_n\\-I_n&&&\\&&I_n&\end{smallmatrix}\right)\in G$,
$\xi_n=(0,\ldots,0,1)\in E^n$ and $V_\gamma=\gamma^{-1}N\gamma\cap V$.
Note that we have the relation of characters $\psi_{N'_{\Levi'}}(u)=\psi_{N_\Levi}(\gamma u\gamma^{-1})$ if $u\in N'_{\Levi'}$.

The basic properties of $\whitform$ are summarized in the following lemma.
Its proof is identical to \cite[Lemmas~4.5 and 4.9]{1401.0198} and will be omitted.

\begin{lemma} \label{L: whitform}
\begin{enumerate}
\item The integral \eqref{eq: defwhitform} is well defined and absolutely convergent uniformly
for $s$ and $g$ in compact sets.
Thus $\whitform(W,\Phi, g,s)$ is entire in $s$ and smooth in $  g$.
In the non-archimedean case the integrand is compactly supported on $V_\gamma\bs V$.
\item For any $W\in\Ind(\WhitML(\pi))$, $\Phi$ and $s\in\C$, the function
$g\mapsto\whitform(W,\Phi,g,s)$ is smooth and $(N',\psi_{N'})$-equivariant.
\item For any $ g,x\in G'$ and $v\in V$ we have
\begin{equation} \label{eq: whitform equivariance}
\whitform(I(s,vx)W,\wevinv^{\chiqe^{-1}}(v  x)\Phi, g,s)=\whitform(W,\Phi,g x,s).
\end{equation}
\item \label{part: unramwhitform} Suppose that $E/F$ is $p$-adic and unramified, $p\ne2$, $\pi$ is unramified, $\psi$ is unramified,
$W^\circ \in\Ind(\WhitML(\pi))$ is the standard unramified vector, and $\Phi_0=1_{\OO_E^n}$ (where $\OO_E$ is the ring of integers in $E$).
Then $\whitform(W^\circ,\Phi_0,e,s)\equiv1$ (assuming $\vol(V\cap K)=\vol (V_\gamma\cap K)=1$).
\end{enumerate}
\end{lemma}

Let $\pi\in\Irr_{\gen}\GLnn$, considered also as a representation of $\Levi$ via $\levi$.
By the same argument as in \cite[Theorem in \S1.3]{MR1671452}, (see also \cite[Remark~4.13]{1401.0198}),
for any non-zero subrepresentation $\pi'$ of $\Ind(\WhitML(\pi))$
there exists $W\in\pi'$ and $\Phi\in\swrz(E^n)$ such that $\whitform(W,\Phi,\cdot,0)\not\equiv0$.

Assume now that $\pi\in\Irr_{\gen,\ut}\GLnn$.
By Proposition \ref{prop: M1/2} $M(s)$ is holomorphic at $s=\frac12$.
Denote by $\des^{\chiqe}(\pi)$ the space of Whittaker functions on $G'$ generated by $\whitform(M(\frac12)W,\Phi,\cdot,-\frac12)$,
$W\in \Ind(\WhitML(\pi))$, $\Phi\in\swrz(E^n)$.
By the above comment $\des^{\chiqe}(\pi)\ne0$.

Let $\pi'$ be the image of $\Pi$ under $M(\frac12)$.
By \eqref{eq: whitform equivariance} the space $\des^{\chiqe}(\pi)$ is canonically a quotient of
the $G'$-module $J_V(\pi'\otimes\wevinv^{\chiqe^{-1}})$ of the $V$-coinvariant of $\pi'\otimes\wevinv^{\chiqe^{-1}}$.
We view $J_V(\pi'\otimes\wevinv^{\chiqe^{-1}})$ as the ``abstract'' descent and $\des^{\chiqe}(\pi)$ as the ``explicit'' descent.

\subsection{Local Shimura integrals}

Now let $\pi\in\Irr_{\gen}\Levi$ and $\sigma\in\Irr_{\gen,\psi_{N'}^{-1}}G'$ with Whittaker model $\WhitGd(\sigma)$.
Following Ginzburg--Rallis--Soudry \cite[\S 10]{MR2848523}, for any $W'\in \WhitGd(\sigma)$, $W\in\Ind(\WhitML(\pi))$ and $\Phi\in\swrz(E^n)$ define
the local Shimura type integral
\begin{equation}\label{eq: localinner}
{J}( W',W,\Phi,s):=\int_{N'\bs G'}W'(g')\whitform(W,\Phi, g,s)\ dg.
\end{equation}

The analogue of these integrals in the symplectic and metaplectic case were studied in detail in \cite{MR3267119}.
The same methods no doubt work in the case at hand.
Unfortunately this has not been carried out in the literature, as far as we know.
Since this is beyond the scope of the current paper we will simply list the expected properties as a working assumption.
(See also \cite[\S 10.6]{MR2848523} and \cite{MR2522026} for some of the claims below.)

\begin{assumption}\label{A: analytic}
Suppose that $\pi\in \Irr_{\gen}\GLnn$. Then
\begin{itemize}
\item ${J}$ converges in some right-half plane (depending only on
$\pi$ and $ \sigma$), admits a meromorphic continuation in $s$.
\item For any $s\in\C$ we can choose $W'$, $W$ and $\Phi$ such that $J(W',W,\Phi,s)\ne0$.
\item If $E/F$ is $p$-adic and unramified, $p\ne 2$, $\pi$, $\sigma$ and $\psi$ are unramified, $W^\circ$ and $\Phi_0$ are as in
Lemma \ref{L: whitform} part \ref{part: unramwhitform} and $W^{'\circ}$ is $K'$-invariant with $ W^{'\circ}(e)=1$ then (see \cite[(10.64)]{MR2848523})
\begin{equation} \label{eq: Junram}
{J}( W^{'\circ},W^\circ,\Phi_0,s)=\vol(K')\frac{L(\sigma\times(\pi\otimes\chiqe^{-1}),s+\frac12)}{L(2s+1,\pi,\As^+)},
\end{equation}
assuming the Haar measures on $V,V_\gamma$ and $N'$ are normalized so that $\vol(V\cap K), \vol(V_\gamma\cap K)$ and $\vol(N'\cap K')$ are all $1$.
\end{itemize}
\end{assumption}

One also expects a functional equation for $J$ as in \cite{MR1671452} and \cite{MR3267119}.
(We will not use it in the paper.)

\subsection{Main reduction theorem}
Let $F$ be a local field.
Let $\pi\in\Irr_{\gen,\ut}\GLnn$. We say that $\pi$ is \emph{good}
if the following conditions are satisfied for all $\psi$:
\begin{enumerate}
\item $\des^{\chiqe}(\pi)$ and $\desinv^{\chiqe^{-1}}(\c{\pi})$ are irreducible.
\item ${J}(W',W,\Phi,s)$ is holomorphic at $s=\frac12$ for any $W'\in \desinv^{\chiqe^{-1}}(\c{\pi})$, $W\in\Ind(\WhitML(\pi))$ and $\Phi\in\swrz(E^n)$.

\item For any $W'\in\desinv^{\chiqe^{-1}}(\c{\pi})$,
\begin{equation} \label{eq: factorsthru}
J(W',W,\Phi,\frac12)\text{ factors through the map }
(W,\Phi)\mapsto(\whitform(M(\frac12)W,\Phi,\cdot,-\frac12)).
\end{equation}
\end{enumerate}
From \eqref{eq: whitform equivariance}:
\[
J(\sigma'(x)W',I(s,vx)W,\wevinv^{\chiqe^{-1}}(v  x)\Phi, s)=J(W',W,\Phi,s)
\]
where $\sigma'=\desinv^{\chiqe^{-1}}(\c{\pi})$.
Thus if $\pi$ is good, there is a non-degenerate $G'$-invariant pairing $[\cdot,\cdot]$ on $\desinv^{\chiqe^{-1}}(\c{\pi})\times\des^{\chiqe}(\pi)$ such that
\[
J(W',W,\Phi,\frac12)=[W',\whitform(M(\frac12)W,\Phi,\cdot,-\frac12)]
\]
for any $W'\in\desinv^{\chiqe^{-1}}(\c{\pi})$, $W\in\Ind(\WhitML(\pi))$ and $\Phi\in\swrz(E^n)$. By \cite[\S2]{MR3267120}, when $\pi$ is good, there exists a non-zero constant $c_\pi$ such that
\begin{equation} \label{eq: unemain}
\stint_{N'}J(\sigma'(u)W',W,\Phi,\frac12)\psi_{N'}(u)\,du
=c_\pi W'(e)\whitform(M(\frac12)W,\Phi,e,-\frac12).
\end{equation}

\begin{remark}\label{rem: indmeas}
Note that a priori $c_\pi$ implicitly depends on the choice of Haar measures on $G'$ and
$U$ (the latter used in the definition of the intertwining operator), but not on any other groups.
However, the Lie algebras of $G'$ and $U$ are identical, both equal to
$$\{X\in \Mat_{2n,2n}(E): \c{X}w_{2n}=w_{2n}X\}.$$
We can thus identify the gauge forms on $G'$ and $U$ and therefore $c_\pi$ does not depend on any choice
if we use the unnormalized Tamagawa measures on $G'$ and $U$ with respect to the same gauge form.
(We will say that the Haar measures on $G'$ and $U$ are compatible in this case.)
\end{remark}

From now on we assume that the measures on $G'$ and $U$ are compatible.
Note that when $E/F$ is unramified and $p$-adic then
\begin{equation}\label{eq: volform}
\vol(K')=\vol(U\cap K)\big(\prod_{j=1}^{2n}L(j,\chiq^j)\big)^{-1}.
\end{equation}

\begin{lemma}\label{L: unramified}
If $E/F$ is $p$-adic and unramified, $p\ne 2$, $\pi$, $\sigma$ and $\psi$ are unramified, $W^\circ$ and $\Phi_0$ are as in
Lemma \ref{L: whitform} part \ref{part: unramwhitform} and $W^{'\circ}$ is $K'$-invariant with $ W^{'\circ}(e)=1$ then \eqref{eq: unemain} holds with $c_\pi=1$.
\end{lemma}
\begin{proof}
By \cite[Proposition~2.14]{MR3267120}, the left-hand side of \eqref{eq: unemain} is (assuming $\vol(N'\cap K')=1$)
\[
\big(\prod_{j=1}^{2n}L(j,\chiq^j)\big)J(W^{'\circ},W^{\circ},\Phi_0,\frac12)/L(1,\sigma,\Ad)
\]
where $\sigma=\des^{\chiqe}(\pi)$.
Since $L(1,\sigma,\Ad)=L(1,\pi,\As^-)$ and $L(s,\pi\otimes\c{\pi})=L(s,\pi,\As^+)L(s,\pi,\As^-)$,
the Lemma follows from \eqref{eq: unramwhit}, \eqref{eq: Junram}, \eqref{eq: volform} and Lemma \ref{L: whitform} part \ref{part: unramwhitform}.
\end{proof}

We can now state our main reduction theorem in the case of $\une$. Let $F$ be a number field.
\begin{theorem} \label{unethm: local to global}
Let $\pi\in\Ucusp\GLnn$ and let $k$ be as above.
Assume our Working Assumptions \ref{uneassume} and \ref{A: analytic}. Then for all $v$ $\pi_v$ is good.
Moreover, let $S$ be a finite set of places including all the archimedean and even places
such that $E/F$, $\pi$ and $\psi$   are unramified outside $S$. Let $\sigma=\des^{\chiqe}(\pi)$.
Then for any $ \varphi\in \sigma$ and $ \varphi^\vee\in\d\sigma$ which are fixed under $K'_v$ (maximal compact of $G'_v$)  for all $v\notin S$ we have
\label{eq: mainidtmodinfty}
\begin{multline} \label{eq: modglobalidentity}
 \whit^{\psi_{N'}}( \varphi) \whit^{\psi_{ N'}^{-1}}(\d\varphi )=
2^{1-k}(\prod_{v\in S}c_{\pi_v}^{-1})\frac{\prod_{j=1}^{2n}L^S(j,\chiq^j)}{L^S(1,\pi,\As^-)}\\
(\vol(N'(\OO_S)\bs N'(F_S)))^{-1}\stint_{N'(F_S)}(\sigma(u) \varphi,\d\varphi)_{G'}\psi_{ N'}(u)^{-1}\ du.
\end{multline}
\end{theorem}
Note that  by Lemma~\ref{L: unramified},  the above statement is independent of the choice of $S$.

\subsection{Proof of Theorem~\ref{unethm: local to global}}
The proof follows the same line of argument as \cite[Theorem~6.2]{1401.0198}. For convenience, all global measures are taken to be Tamagawa measures.

From \cite[Theorem~9.7 (1)]{MR2848523} and the same procedure described in the proof of \cite[Theorem~6.3]{1401.0198}, we get the following identity:
\begin{proposition} \label{prop: whitdesc}
Let $\pi\in \Ucusp\GLnn$ and $\varphi\in \AF(\pi)$, $\Phi\in \swrz(\A_E^n)$,
\[
\whit^{\psi_{N'}}(\FJ(\reseisen\varphi,\Phi),g)=
\int_{V_\gamma(\A)\bs V(\A)}\whit^{\psi_N}(\reseisen\varphi, \gamma   v  g)
\wevinv^{\chiqe^{-1}}(v g)\Phi(\xi_n)\,dv, \ \ g\in G'(\A_E).
\]
Here $\whit^{\psi_{N'}}(\phi,\cdot)=\int_{N'(F)\bs N'(\A)}\phi(u\cdot)\psi_{N'}^{-1}(u)\ du$. $\qed$
\end{proposition}

By Proposition \ref{prop: whitdesc}, formula \eqref{eq: unramwhit} and Lemma \ref{L: whitform} part \ref{part: unramwhitform},
for any factorizable $\varphi\in\AF(\pi)$ we have (for $S$ large enough)
\begin{equation}\label{uneeq: whitdesc}
 \whit^{\psi_{N'}}(\FJ(\reseisen\varphi,\Phi),g)=\resm^S(\pi)
\prod_{v\in S}\whitform_v(M_v(\frac12) W_v,\Phi_v,g_v,-\frac12)
\end{equation}
for any $\Phi=\otimes_v\Phi_v\in\swrz(\A_E^n)$ where
$\whit^{\psi_{N_\Levi}}(\varphi,\cdot)=\prod_v W_v$ and
\[
\resm^S(\pi)=\lim_{s\rightarrow\frac12}(s-\frac12)^k\frac{L^S(2s,\pi,\As^+)}{L^S(2,\pi,\As^+)}.
\]
This factorization together with Assumption~\ref{uneassume} gives the irreducibility of the local descent $\des^{\chiqe}(\pi_v)$ when $\pi_v$ is a local component of $\pi$.

Meanwhile, let $\sigma'$ be an irreducible $\psi_{N'}^{-1}$-generic cuspidal automorphic representation of $\une$.
From \cite[Theorem 10.4, (10.6), (10.64)]{MR2848523} we get:
\begin{proposition}\label{prop: innerune}
Suppose that $\varphi'\in \sigma'$ with $\whit^{\psi_{N'}^{-1}}(\varphi')=\prod_v W_v'$,
$\Phi=\otimes\Phi_v$ and $\whit^{\psi_{N_\Levi}}(\varphi,\cdot)=\prod_v W_v$.
Then for any sufficiently large finite set of places $S$ we have
\begin{multline}\label{uneeq: globalinner}
(\varphi',\FJ(\eisen(\varphi,s),\Phi))_{G'}=\\
\frac12\big(\prod_{j=1}^{2n}L^S(j,\chiq^j)\big)^{-1}
\frac{L^S(s+\frac12,\sigma'\otimes\pi\otimes\chiqe^{-1})}{L^S(2s+1,\pi,\As^+)}
\prod_{v\in S}J_v( W_v',W_v,\Phi_v,s).
\end{multline}
Here on the right hand side we take the unnormalized Tamagawa measures on $G'(F_S)$, $N'(F_S)$, $V(F_S)$ and $V_\gamma(F_S)$
(which are independent of the choices of gauge forms when $S$ is sufficiently large).
\end{proposition}

Note that the volume of $G'(F)\bs G'(\A)$, appearing on the right-hand side of \eqref{eq: innerdef}, is equal to $2$ when we use Tamagawa measure as in \cite{MR2848523}.

Let $\pi\in\Ucusp\GLnn$ and $\sigma'=\desinv^{\chiqe^{-1}}(\c{\pi})$.
By Proposition \ref{prop: ueven} $\d{\sigma'}$ weakly lifts to $\pi\otimes\chiqe^{-1}$.
Thus, we have $L^S(s,\sigma'\otimes(\pi\otimes\chiqe^{-1}))=L^S(s,\pi\otimes\d\pi)=L^S(s,\pi\otimes\c{\pi})$, which has a pole of order $k$ at $s=\frac12$.
By our Working Assumptions, $ {J}_v(  W'_v,W_v,\Phi_v,s)$ is non-vanishing at $s=\frac12$ for suitable $W_v$, $ W'_v$ and $\Phi_v$.
Thus the right-hand side of \eqref{uneeq: globalinner} (when $ \varphi'\in\sigma'$)
has a pole of order at least $k$ at $s=\frac12$ for suitable $\varphi$, $ \varphi'$ and $\Phi$.
On the other hand, the left-hand side of \eqref{uneeq: globalinner} has a pole of order $\le k$ because this is true for
$\eisen(\varphi,s)$ and $\varphi'$ is rapidly decreasing.
We conclude that ${J}_v( W'_v,W_v,\Phi_v,s)$ is holomorphic at $s=\frac12$ for all $v$.

Multiplying \eqref{uneeq: globalinner} by $(s-\frac12)^k$ and taking the limit as $s\rightarrow\frac12$, we get
for $\varphi\in \AF(\pi)$, $\Phi\in\swrz(\A_E^n)$  and $\varphi'\in \sigma'$:
\begin{multline} \label{uneeq: FJinnerprod}
(\varphi',\FJ(\reseisen\varphi,\Phi))_{G'}=\\
\frac12\big(\prod_{j=1}^{2n}L^S(j,\chiq^j)\big)^{-1}
\frac{\lim_{s\rightarrow1}(s-1)^kL^S(s,\pi\otimes\c{\pi})}{L^S(2,\pi,\As^+)}
\prod_{v\in S}J_v(W_v',W_v,\Phi_v,\frac12).
\end{multline}

To show that \eqref{eq: factorsthru} holds for $\pi_{v_0}$ at a fixed place $v_0$,
assume that $W_{v_0}$ and $\Phi_{v_0}$ are such that $\whitform(M_{v_0}(\frac12)W_{v_0},\Phi_{v_0},\cdot,-\frac12)\equiv 0$.
Then by \eqref{uneeq: whitdesc}  $\whit^{\psi_{N'}}(\FJ(\reseisen\varphi,\Phi),\cdot)\equiv 0$ and therefore $\FJ(\reseisen\varphi,\Phi)\equiv0$
by the irreducibility and genericity of the descent. By \eqref{uneeq: FJinnerprod} we conclude that $J_{v_0}( W'_{v_0},W_{v_0},\Phi_{v_0},\frac12)=0$.
Thus, $\pi_{v_0}$ is good.

The identity \eqref{eq: modglobalidentity} follows from \eqref{uneeq: whitdesc}, \eqref{uneeq: FJinnerprod}, \eqref{eq: unemain} and Lemma~\ref{L: unramified}.
This conclude the proof of Theorem \ref{unethm: local to global}.

\subsection{Local conjecture}

It follows from Theorem~\ref{unethm: local to global} and Proposition \ref{prop: ueven}
that (under our work assumptions) Conjecture \ref{conj: global} for $\une$  is a consequence of the following local conjecture:
\begin{conjecture} \label{uneconj: main}
Let $\pi\in \Irr_{\gen,\ut}\GLnn$ be unitarizable and good.
Then (for compatible Haar measures on $G'$ and $U$) we have $c_{\pi}=\omega'_\pi(-1)$.
\end{conjecture}

\section{A heuristic argument: case of $\U_2$} \label{sec: n=1}

We will substantiate Conjecture \ref{uneconj: main} by giving a heuristic argument in the case $n=1$.
Namely, we will manipulate the integrals formally as if they were absolutely convergent.

Throughout let $F$ be a local field. We first consider the case where $E/F$ is inert.
We use the following measures. On $E$ and $F$ we take the self-dual Haar measures with respect to $\psi$ and $\psi\rest_F$ respectively.
These determine multiplicative Haar measures on $E^\times$ and $F^\times$.
On $G'$, using the Bruhat decomposition we take $dg=dx\ dy \ dt$ where
$g=\sm{t}{}{}{\c{t}^{-1}}\sm{1}{x}{}{1}\sm{1}{}{y}{1}$ with $x,y\in F$, $t\in E^{\times}$.
On $U$, we take the Haar measure $du=dx\ dy \ dz$ where we write
$u=\left(\begin{smallmatrix}1&&x&y\\&1&z&\c{x}\\&&1&\\&&&1\end{smallmatrix}\right)$ with $y,z\in F$ and $x\in E$.
Note that these measures are compatible in the sense of  Remark~\ref{rem: indmeas}.


We follow the arguments of \cite[\S7]{1401.0198} skipping the similar calculations and emphasizing the differences.

For $\Phi\in\swrz(E)$ and $f\in C^\infty(G)$ (recall $G=U_4^-$) define
\[
\Phi*f(g)=\int_E f(g\left(\begin{smallmatrix}1&r&&\\&1&&\\&&1&{-\c{r}}\\&&&1\end{smallmatrix}\right))\Phi(r)\ dr.
\]
Similarly to \cite[``Claim 7.1"]{1401.0198} we get the following formula for $\whitform(W,\Phi,\cdot,s)$ on the big cell:


\begin{claim} Let $g=\diag(1,\sm{}1{-1}{}\sm{t}{tz}{}{\c{t}^{-1}},1)\in G'\subset G$, ($t\in E^\times$, $z\in F$). Then
\begin{multline} \label{eq: whitn1}
\whitform(W,\Phi,g,s)
=\\\abs{t}^{\frac12}\chiqe^{-1}(t)
\int_F\int_{E}(\Phi*(W_s))(\left(\begin{smallmatrix}t^{-1}&&&\\&1&&\\&&1&\\&&&\c{t}\end{smallmatrix}\right)
\left(\begin{smallmatrix}1&&&\\&1&&\\x&y&1&\\-z&\c{x}&&1\end{smallmatrix}\right)
\left(\begin{smallmatrix}&&1&\\&&&1\\{-1}&&&\\&-1&&\end{smallmatrix}\right))\psi(\frac12 y)\,dx\,dy.
\end{multline}
\end{claim}

Next we give a description of the two sides of \eqref{eq: unemain}.
For simplicity for $W\in \Ind(\WhitML(\pi))$ we set $M^*W:=(M(\frac12)W)_{-\frac12}$.

\begin{claim}\label{claim: phi} (See \cite[``Claim 7.2"]{1401.0198})
\begin{enumerate}
\item \label{claim: whitfunct}
For any $W\in\Ind(\WhitML(\pi))$ we have $\whitform(W,\Phi,e,s)=\Mint(\Phi *(W_s))$ where
\begin{equation}\label{eq: defmintone}
\Mint(W):=
\int_F W (\left(\begin{smallmatrix}1&&&\\&1&&\\&x&1&\\&&&1\end{smallmatrix}\right)
\left(\begin{smallmatrix}&1&&\\&&&1\\-1&&&\\&&1&\end{smallmatrix}\right)
\left(\begin{smallmatrix}1&-1&&\\&1&&\\&&1&1\\&&&{1}\end{smallmatrix}\right) )\psi(\frac12 x)\, dx.
\end{equation}

\item \label{claim: mainleft}
Let $W'=\whitformd(M(\frac12)\alt{W},\d{\Phi},\cdot,-\frac12)$ for some
$\alt{W}\in \Ind(\WhitMLd(\c{\pi}))$ and $\d{\Phi}\in \swrz(E)$.
Then the left-hand side of \eqref{eq: unemain} equals
\begin{equation} \label{one: beforefe}
\int_{E^\times} I^{\psi}(\Phi *(W_{\frac12});\sm{t}{}{}{1})I^{\psi^{-1}}(\d{\Phi}*M^*\alt{W};\sm{t}{}{}{1})\abs{t}^{-2}\,dt
\end{equation}
where for any function $f\in C^{\infty}(\U_4^-)$ we set
\[
I^{\psi}(f;g):=\iint_{F^2}\int_E f (\sm{g}{}{}{g^*}
\left(\begin{smallmatrix}1&&&\\&1&&\\x&y&1&\\z&\c{x}&&1\end{smallmatrix}\right)
\left(\begin{smallmatrix}&&1&\\&&&1\\-1&&&\\&-1&&\end{smallmatrix}\right) )\\\psi(\frac{y-z}{2})\ dx\ dy\ dz,\ \ g\in\GL_2(E).
\]
\end{enumerate}
\end{claim}

From ``Claim"~\ref{claim: phi}, we are left to show the identity
\begin{equation}\label{goal: n=1}
\int_{E^\times} I^{\psi}(W_{\frac12};\sm{t}{}{}{1})I^{\psi^{-1}}(M^*\alt{W};\sm{t}{}{}{1})\abs{t}^{-2}\,dt=
\Mint(M^*W)\Mintd(M^*\alt{W})
\end{equation}
for any $(W,\alt{W})\in \Ind(\WhitML(\pi))\times\Ind(\WhitMLd(\c{\pi}))$.

For any $g\in G$ the function
\[
m\mapsto W_s(\left(\begin{smallmatrix}m&\\&m^*\end{smallmatrix}\right)g)\abs{\det m}^{-(s+1)}
\]
belongs to $\WhitM(\pi)$ and therefore the integral in \eqref{goal: n=1} over $t$ (with the rest of the variables fixed)
has the form
\[
\int_{E^\times} W^1(\sm{t}{}{}{1})W^2(\sm{t}{}{}{1})\ dt
\]
where $W^1\in\WhitM(\pi)$ and $W^2\in\WhitMd(\d\pi)$. The key
observation is that the above integral defines a $\GL_2(E)$-invariant
bilinear form on $\WhitM(\pi)\times\WhitMd(\d\pi)$, and thus
\begin{equation}\label{eq: fen=1}
\int_{E^\times} W^1(\sm{t}{}{}{1})W^2(\sm{t}{}{}{1})\ dt=\int_{E^\times} W^1(\sm{t}{}{}{1}b) W^2(\sm{t}{}{}{1}b)\ dt
\end{equation}
for any $b\in \GL_2(E)$.
To show \eqref{goal: n=1} we only need to prove
\begin{equation}\label{goal2: n=1}
\int_{E^\times} I^{\psi}(W_{\frac12};\sm{t}{}{}{1}b)I^{\psi^{-1}}(M^*\alt{W};\sm{t}{}{}{1}b)\abs{t}^{-2}\abs{\det b}^{-2}\,dt=
\Mint(M^*W)\Mintd(M^*\alt{W})
\end{equation}
for a well-chosen $b\in \GL_2(E)$.

Fix $\rr\in E^{\times}$ such that $\c{\rr}=-\rr$. 
Let $b=\sm{\rr}{}{}{1}\sm{-\frac12}{\frac12}{1}{1}$. 
Making a change of variables, we get
\begin{equation} \label{eq: I0afterb}
I^{\psi}(f;\sm{t}{}{}{1}b)=\abs{\rr}^2
\iint_{F^2}\int_E f(\left(\begin{smallmatrix}t &&&\\&1&&\\&&1&\\&&&\c{t}^{-1}\end{smallmatrix}\right)
\left(\begin{smallmatrix}1&&&\\&1&&\\ x&y&1&\\z &\c{x}&&1\end{smallmatrix}\right)\sm b{}{}{b^*}
\left(\begin{smallmatrix}&&1&\\&&&1\\-1&&&\\&-1&&\end{smallmatrix}\right))\\\psi(\rr x)\ dx\ dy\ dz.
\end{equation}

Since $\pi\in\Irr_{\ut,\gen}\GL_2(E)$, the linear form $w\mapsto \int_{F^\times}W(\sm {t\tau}{}{}1)\ dt$ on $\whit^{\psi_{N_M}}(\pi)$ is
$\GL_2(F)$-invariant. (This is true more generally -- see \cite{MR2787356}
in the $p$-adic case and \cite{1212.6436} in the archimedean case.)
We will use this to define a nontrivial $H$-invariant linear form on the image of $M^*$
(recall that $H=\Sp_2(F)=\U_4^-\cap \GL_4(F)$).
Namely, for any $W\in\Ind(\WhitML(\pi))$ define (see \S\ref{sec: inflate})
\begin{equation}\label{eq: defLWn=1}
L_W(g)=\int_{F^\times}\iint_{F^3}W_\frac12 (
\left(\begin{smallmatrix}t\rr&&&\\&1&&\\&&1&\\&&&\c{t\rr}^{-1}\end{smallmatrix}\right)
\left(\begin{smallmatrix}1&&&\\&1&&\\x&y&1&\\z&x&&1\end{smallmatrix}\right)g )\ dz\ dy\ dx\abs{t}^{-\frac32}\ dt.
\end{equation}
Then $L_W$ is left $H$-invariant.
Note that $L_W$ has the extra equivariance property
$$L_W(\diag(z,z,\c{z}^{-1},\c{z}^{-1})g)=\omega_\pi(z)L_W(g).$$ 
Using Fourier inversion, we can express $M^*W$ in terms of $L_W$ as follows. (See \cite[``Claim 7.4"]{1401.0198})

\begin{claim} \label{claim: M12}
For any $W\in\Ind(\WhitML(\pi))$, 
\[
M^*W(g)=\abs{\rr}^{-\frac32}\iint_{F^2} L_W(\left(\begin{smallmatrix}1&&&\\&\c{\rr}&&\\&&\rr^{-1}&\\&&&1\end{smallmatrix}\right)
\left(\begin{smallmatrix}1&s&r&\\&1&&r\\&&1&-s\\&&&1\end{smallmatrix}\right)g)\psi(s)\ dr\ ds.
\]
\end{claim}

Applying this claim together with Fourier inversion, we get:
\begin{claim}\label{claim: main}
For any $W\in\Ind(\WhitML(\pi))$ and $\phi\in C^\infty(E^\times)$ we have
\begin{equation}\label{one: goal}
\int_{E^\times} I^{\psi}(M^*W;\sm{t}{}{}{1}b)\phi(t)\ dt= \abs{\rr} \Mint(M^*W)\int_{F^\times} \abs{t}^{\frac12}\phi(t\rr)\ dt
\end{equation}
and
\begin{equation}\label{one: factor}
\int_{F^\times}I^\psi(W_\frac12;\sm{t\rr}{}{}{1}b)\abs{t}^{-\frac32}\ dt=\abs{\rr}^{3}\omega_\pi(\rr)\Mint(M^*W).
\end{equation}
Meanwhile,
$$
\Mint(M^*W)=\abs{\rr}^{-\frac32}\omega_\pi(\rr)\int_F L_W(\left(\begin{smallmatrix}1&& &\\&1&& \\
x/\rr&&1&\\&x/\c{\rr}&&1\end{smallmatrix}\right)
\sm b{}{}{b^*}
\left(\begin{smallmatrix}&&1&\\&&&1\\-1&&&\\&-1&&\end{smallmatrix}\right))\psi( x) \ dx.
$$
\end{claim}

The argument for the above claim is similar to that of \cite["Claim" 7.5,7.6]{1401.0198} and therefore will be skipped.
The above claim implies \eqref{goal2: n=1}. This concludes the heuristic argument for the conjecture $c_\pi=\omega_\pi(\rr)=\omega'_\pi(-1)$.

Finally in the case $E=F\oplus F$, we can follow the above argument and replace $\rr$ by the element $(1,-1)$, $\abs{\rr}$ by $1$.
In this case $\pi$ can be identified with a pair $(\pi_1,\d{\pi_1})$, then $\omega'_\pi$ is the character on the group
$\{(a,a^{-1})\}$ given by $\omega'(a,a^{-1})=\omega_{\pi_1}(a)$. Again we get $c_\pi=\omega_{\pi_1}(-1)$ and thus the conjectural identity.

\part{The case $\uno$}

\section{Gelfand--Graev coefficients and descent}
\subsection{Notation} \label{unosec: notation}
We fix the notation in the local setting. Let $G=\biguno=\{g\in\GL_{4n+2}(E):\, \c{\,^tg}w_{4n+2}g=w_{4n+2}\}$. Thus $\GLnn=\GL_{2n+1}$.
Let $G'\simeq\uno\subset G$ consisting of elements fixing $e_1,\ldots,e_n,e_{-1},\ldots,e_{-n}$ and $e_{2n+1}+e_{-1-2n}$.
(Note that we take $\alpha=-2$ in the notation of \cite[(3.40)]{MR2848523}.)

For $x\in \Mat_{l,m}$, let $\startran{x}$ be the matrix in $\Mat_{m,l}(E)$ given by $\startran{x}=w_m \c{\,^tx} w_l$.
Set $\antisym_n=\{x\in\Mat_{n,n}:\startran{x}=-x\}$.

Let $\GLn=\GL_n$ and let $\levi'$ be the embedding of $\GLn$ in $G'$ given by $\levi'(m)=\diag(I_n,m,I_2,m^*, I_n)$.
Let $U'$ be the Siegel unipotent subgroup of $G'$. Then $U'=U'_1\ltimes U'_0$ with
\[
U'_0=\{\diag(I_n, \begin{pmatrix} I_n&&u\\&I_2&\\&&I_n\end{pmatrix}, I_n):u\in\antisym_n\},
\]
and
\[
U'_1=\{\diag( I_n,\begin{pmatrix} I_n&(v,-v)&-2v\startran{v}\\&I_2&\left(\startran{v}\atop {-\startran{v}}\right)\\&&I_n
\end{pmatrix}, I_n):v\in E^n\}.
\]

Let $N$ be the standard maximal unipotent subgroup of $G$ and
let $N'=N\cap G'$, a maximal unipotent subgroup of $G'$. Then $N'=\levi'(N'_{\GLn})\ltimes U'$ where $N'_{\GLn}$ is the standard maximal unipotent subgroup of $\GLn$.
Let $V$ be the unipotent radical of the standard parabolic subgroup of $G$ with Levi $\GL_1^n\times\U^+_{2n+2}$.

Let $N_{\GLnn}$ be the standard maximal unipotent subgroup of $\GLnn$.
We fix a non-degenerate character $\psi_{N_\GLnn}$ of $N_\GLnn$ and set $\psiweil(x)=\psi_{N_\GLnn}(I_{2n+1}+x\one_{n+1,n+2}^{2n+1})$
and $\psiweil'(x)=\psi_{N_\GLnn}(I_{2n+1}+x\one_{n,n+1}^{2n+1})$.
We assume that $\psiweil(x)=\psiweil(\c{x})^{-1}$ and $\psiweil'(x)=\psiweil'(\c{x})^{-1}$.
As in the even case, the results stated below are independent of the choice of $\psi_{N_\GLnn}$.
The character $\psi_{N_\GLnn}$ will determine characters of several other unipotent group as follows:
\begin{itemize}
\item $\psi_{N_\Levi}$ is the non-degenerate character of $N_\Levi=\levi(N_\GLnn)$ such that $\psi_{N_\Levi}(\levi(u))=\psi_{N_\GLnn}(u)$.

\item $\psi_{N'_\GLn}$ is the non-degenerate character of $N'_{\GLn}$ given by $\psi_{N'_\GLn}(u)=\psi_{N_\GLnn}(\diag(u,I_{n+1}))$.

\item $\psi_{N'_{\Levi'}}$ is the non-degenerate character of $N'_{\Levi'}=\levi'(N'_{\GLn})$ such that $\psi_{N'_{\Levi'}}(\levi'(u))=\psi_{N'_{\GLn}}(u)$.

\item $\psi_{U'}$ is the character on $U'$ given by $\psi_{U'}(u)=\psiweil'({u_{2n,2n+1}})$ if $u\in U'$.

\item $\psi_{N'}$  of  $N'$ is the non-degenerate character
$\psi_{N'}(nu)=\psi_{N'_{\Levi'}}(n)\psi_{U'}(u)$ with $n\in N'_{\Levi'}$ and $u\in U'$.

\item $N=N_\Levi\ltimes U$ where $U$ is the Siegel unipotent of $G$; $\psi_N$ is the \emph{degenerate} character on $N$ given by
$\psi_N(uv)=\psi_{N_\Levi}(u)$ for any $u\in N_\Levi$ and $v\in U$.

\item An element in $V$ can be written as $vu$ where $u$ fixes $e_1,\ldots e_n$, $v$ fixes $e_{n+1},e_{n+2},\ldots,e_{-n-1}$.  Set
$$\psi_V(vu)=\psi_{N_\Levi}(w_U v w_U^{-1})\psiweil^{-1}(\c{u_{n,2n+1}+u_{n,2n+2}}).$$
\end{itemize}

For convenience, we will fix a non-trivial character $\psi$ of $E$, and further assume
$\psi(x)=\psi(\c{x})^{-1}$. Set
\[
\psi_{N_\GLnn}(u)=\psi(u_{1,2}+\dots+u_{2n-1,2n}).
\]
Thus $\psiweil=\psiweil'=\psi$.

The notation introduced has an obvious global counterpart.

\subsection{Descent map}
We now go to the global setting.
Let $\Ucusp \GLnn$  be as defined in \S\ref{unesec: desc}. For $\pi\in \Ucusp \GLnn$ and $\varphi\in \AF(\pi)$, the associated
Eisenstein series $\eisen(\varphi,s)$ has a pole of order $k$ at $s=\frac12$.
Let $\reseisen\varphi=\lim_{s\rightarrow\frac12}(s-\frac12)^k\eisen(\varphi,s)$.
For an automorphic form $f$ on $G(\A)$, let $\GG(f)$ be the Gelfand--Graev coefficient (a function on $G'(\A)$)
\begin{equation} \label{def: GG}
\GG(f)(g)=\int_{V(F)\bs V(\A)}f(vg)\psi_V^{-1}(v)\,dv,\ \ g\in G'(\A).
\end{equation}
By definition, the \emph{descent} of $\pi$ (with respect to $\psi_{N_\GLnn}$) is the space
$\desc_{\psi_{N_\GLnn}}(\pi)$ generated by $\GG(\reseisen\varphi)$ with $\varphi\in\AF(\pi)$. It is known that $\desc_{\psi_{N_\GLnn}}(\pi)$ is
cuspidal, multiplicity free and generic (\cite[Theorem~3.1]{MR2848523}).  As in the case of $\une$, we make the following assumption:
\begin{assumption}\label{unoassume}
When $\pi \in\Ucusp\GLnn$, $\desc_{\psi_{N_\GLnn}}(\pi)$ is irreducible.
\end{assumption}
It then follows from \cite[Theorem~11.2]{MR2848523} that
\begin{proposition}\label{prop: uodd}
With the above assumption, $\pi\mapsto \sigma=\desc_{\psi_{N_\GLnn}}(\pi)$ defines a bijection between $\Ucusp\GLnn$ and $\Cusp_{\gen} G'$
the set of generic irreducible cuspidal automorphic representations of $G'$. Moreover $\d\sigma$ weakly lifts to $\c{\pi}$.
\end{proposition}

\section{Reduction to a local conjecture}

\subsection{Local descent}
Locally, for any $W\in\Ind(\WhitML(\pi))$  define a function on $G'$:
\begin{equation}\label{unoeq: defwhitform}
\whitformo(W,g,s)=\int_{V_\gamma\bs V} W_s(\gamma v  g)\psi_V^{-1}(v)\,dv,\,\,\,g\in G'
\end{equation}
where $\gamma=\left(\begin{smallmatrix}&I_{n+1}&&\\&&&I_n\\I_n&&&\\&&I_{n+1}&\end{smallmatrix}\right)\in G$,
$V_\gamma=V\cap\gamma^{-1}N\gamma$.
The basic properties of $\whitformo$ are summarized in the following lemma whose proof we once again leave out.
(It is very close to the argument in \cite[\S4]{1401.0198}.)

\begin{lemma} \label{L: whitformGG}
\begin{enumerate}
\item The integral \eqref{unoeq: defwhitform} is well defined and absolutely convergent uniformly
for $s$ and $g$ in compact sets.
Thus $\whitformo(W, g,s)$ is entire in $s$ and smooth in $  g$.
In the non-archimedean case the integrand is compactly supported on $V_\gamma\bs V$.
\item For any $W\in\Ind(\WhitML(\pi))$  and $s\in\C$, the function
$ g\mapsto\whitformo(W,  g,s)$ is smooth and $(N',\psi_{N'})$-equivariant.
\item For any $ g\in G'$ we have
\begin{equation} \label{eq: whitform equivarianceGG}
\whitformo(I(s,x)W, g,s)=\whitformo(W,g x,s).
\end{equation}
\item \label{part: unramwhitformGG} Suppose that $E/F$ is $p$-adic and unramified, $\pi$ is unramified, $\psi$ is unramified and
$W^\circ \in\Ind(\WhitML(\pi))$ is the standard unramified vector. Then $\whitformo(W^\circ,e,s)\equiv1$ (assuming $\vol(V\cap K)=\vol(V_\gamma\cap K)=1$).
\end{enumerate}
\end{lemma}

Let $\pi\in\Irr_{\gen,\ut}\GLnn$.
By Proposition \ref{prop: M1/2} $M(s)$ is holomorphic at $s=\frac12$.
Denote by $\des(\pi)$ the space of Whittaker functions on $G'$ generated by $\whitformo(M(\frac12)W,\cdot,-\frac12)$,
$W\in \Ind(\WhitML(\pi))$.
As in the $\une$ case,  $\des(\pi)\ne0$. We call $\des(\pi)$ the (explicit) descent of $\pi$.

\subsection{Whittaker function of descent}
We have the following analogue of Proposition~\ref{prop: whitdesc}:

\begin{proposition}(reformulation of \cite[Theorem 9.5, part (1)]{MR2848523}) \label{thm: FCdescentGG}
Let $\pi\in \Ucusp\GLnn$ and $\varphi\in \AF(\pi)$, we have
\[
\whit^{\psi_{N'}}(\GG(\reseisen\varphi),g)=\vol(V_\gamma(F)\bs V_\gamma(\A))
\int_{V_\gamma(\A)\bs V(\A)}\whit^{\psi_N}(\reseisen\varphi,\gamma v g)
\psi_V^{-1}(v) \,dv, \ \ g\in G'(\A)
\]
and the integral is absolutely convergent.
\end{proposition}

\begin{proof}
It is enough to prove the required identity for $g=e$. We use Tamagawa measures in the proof.
The expression for $\whit^{\psi_{N'}}(\GG(\reseisen\varphi),e)$
in \cite[Theorem 9.5, part (1)]{MR2848523} is (with $\alpha=-2$ and $\lambda=1$ in their notation):
\begin{equation}\label{eq: GRSwhitGG}
\int_{Y(\A)}\big(\int_{X(\A)}\whit^{\psi_N}(\reseisen\varphi,\toUbar(x)\delta\epsilon\kappa y)\ dx\big)\ dy
\end{equation}
where we use the following notation
\begin{itemize}
\item $\delta$ is the Weyl element such that
$ \delta_{i,2i-1}=1$, $i=1,\dots,2n+1$, $\delta_{2n+1+i,2i}=1$, $i=1,\dots,2n+1$,
\item $\kappa=\levi(\kappa')$ where $\kappa'$ is the Weyl element of $\GLnn$ such that
$\kappa'_{2i,i}=1$, $i=1,\dots,n$; $\kappa'_{2i-1,n+i}=1$, $i=1,\dots,n+1$,
\item $\epsilon=\diag(A,\dots,A,I_2,A^*,\dots,A^*)$ where $A=\sm {\frac12}1{-\frac12}1$,\footnote{The definition of $\epsilon$ in
\cite[(9.58)]{MR2848523} is incorrect.}
\item $X$ is the subspace of $\antisym_{2n+1}$ consisting of the strictly upper triangular matrices. $\toUbar(x)=\sm{I_{2n+1}}{}{x}{I_{2n+1}}$
for $x\in \antisym_{2n+1}$.
\item $Y$ is the subgroup of $N_\Levi$ consisting of the matrices of the form
$\levi(\sm{I_n}{y}{}{I_{n+1}})$ where $y$ is lower triangular, (namely $y_{i,j}=0$ if $j>i$).
\end{itemize}
(We note that our choice of the character $\psi_{N_\GLnn}$ is consistent with the choice of \cite{MR2848523} in the case $G=\biguno$.)

Writing
\[
A=\sm 1101\sm 10{-\frac12}1
\]
we get
$\epsilon=\epsilon_U\epsilon_{\bar U}$
where
\[
\epsilon_U=\diag(\sm 1101,\dots,\sm 1101,I_2,\sm 1{-1}01,\dots,\sm 1{-1}01),
\]
\[
\epsilon_{\bar U}=\diag(\sm 10{-\frac12}1,\dots,\sm 10{-\frac12}1,I_2,\sm 10{\frac12}1,\dots,\sm 10{\frac12}1).
\]

Note that
$\delta\epsilon_U\delta^{-1}=\sm{I_{2n+1}}{-\epsilon_\delta}0{I_{2n+1}}$ where $\epsilon_\delta=\diag(-I_n,0,I_n)$.
For any $x\in X(\A)$ we can write
\[
\sm{I_{2n+1}}0x{I_{2n+1}}\sm{I_{2n+1}}{-\epsilon_\delta}0{I_{2n+1}}=\sm{(I_{2n+1}-x\epsilon_\delta)^*}{-\epsilon_\delta}{}{I_{2n+1}-x\epsilon_\delta}
\sm{I_{2n+1}}0{x'}{I_{2n+1}}
\]
where $x'=(I_{2n+1}-x\epsilon_\delta)^{-1}x=x+x\epsilon_\delta x+\dots+(x\epsilon_\delta)^{2n}x$.
After a change of variable $x'\mapsto x$ we get
\[
\int_{Y(\A)}\big(\int_{X(\A)}\whit^{\psi_N}(\reseisen\varphi,
\toUbar(x)\delta\epsilon_{\bar U}\kappa y)\psi^{-1}(x_{n,n+1})\ dx\big)\ dy
\]
since $\psi_{N}(\levi((I_{2n+1}-x\epsilon_\delta)^*))=\psi^{-1}(x_{n,n+1})$.

Note that $\kappa^{-1}\epsilon_{\bar U}\kappa\in Y$ and $\delta\kappa=\gamma$.
Changing variable  $y\mapsto (\kappa^{-1}\epsilon_{\bar U}\kappa)^{-1}y$, we get that the above equals
\[
\int_{Y(\A)}\big(\int_{X(\A)}\whit^{\psi_N}(\reseisen\varphi,
\toUbar(x)\gamma y)\psi^{-1}(x_{n,n+1})\ dx\big)\ dy
\]
Note that
\[
\gamma^{-1}\toUbar(X)\gamma=\{\left(\begin{smallmatrix}I_n&(x_1,x_1')&&x_2\\&I_{n+1}&&\\&&I_{n+1}&*\\&&&I_n\end{smallmatrix}\right):
x_1\in\Mat_{n,n}\text{ strictly upper triangular}, x_2\in\antisym_n, x_1'\in E^n\},
\]
and thus $(\gamma^{-1}\toUbar(X)\gamma Y)\rtimes V_\gamma=V$.
In conclusion we obtain
\[
\int_{V_\gamma(\A)\bs V(\A)}\whit^{\psi_{N}}(\reseisen\varphi,\gamma x)\psi_V^{-1}(x)\ dx
\]
provided it converges, since $\psi_V(\gamma^{-1}\toUbar(x)\gamma)=\psi(x_{n,n+1})$ for $x\in X(\A)$.

Finally, the absolute convergence follows from Lemma \ref{L: whitformGG} applied to $\whitformo(W,g,s)$ in our setting.
\end{proof}

\subsection{Local Shimura integrals}

Now let $\pi\in\Irr_{\gen}\Levi$ and $\sigma\in\Irr_{\gen}G'$ with Whittaker model $\WhitGd(\sigma)$.
As in the $\une$ case,  for any $W'\in \WhitGd(\sigma)$, $W\in\Ind(\WhitML(\pi))$ define
the local Shimura type integral
\begin{equation}\label{eq: localinneruno}
{J}( W',W,s):=\int_{N'\bs G'}W'(g')\whitformo(W, g,s)\ dg.
\end{equation}
Once again we postulate the following working assumptions on the expected analytic properties of the Shimura integral.

\begin{assumption}\label{A: analyticuno}
Suppose that $\pi\in \Irr_{\gen}\GLnn$.
\begin{itemize}
\item ${J}$ converges in some right-half plane (depending only on
$\pi$ and $ \sigma$), admits a meromorphic continuation in $s$.
\item For any $s\in\C$ we can choose $W'$ and $W$  such that $J(W',W,s)\ne0$.
\item If $E/F$ is $p$-adic and unramified,  $\pi$, $\sigma$ and $\psi$ are unramified, $W^\circ$  is as in
Lemma \ref{L: whitform} part \ref{part: unramwhitformGG} and $W^{'\circ}$ is $K'$-invariant with $ W^{'\circ}(e)=1$ then (see \cite[(10.61)]{MR2848523})
\begin{equation} \label{eq: Junramuno}
{J}( W^{'\circ},W^\circ,s)=
\vol(K')\frac{L(s+\frac12,\sigma\times\pi)}{L(2s+1,\pi,\As^+)},
\end{equation}
(assuming $\vol(N'\cap K')=\vol(V\cap K)=\vol(V_\gamma\cap K)=1$).
\end{itemize}
\end{assumption}

One also expects a local functional equation for $J$ similar to \cite{MR1671452} and \cite{MR3267119}.

Similarly to Proposition~\ref{prop: innerune}, from \cite[Theorem 10.3, (10.4), (10.61)]{MR2848523} we get:
\begin{proposition}\label{prop: inneruno}
If $\varphi\in \AF(\pi)$ is such that $\whit^{\psi_{N_\Levi}}(\varphi,\cdot)=\prod_v W_v$,  and $\varphi'\in \sigma'\in \Cusp_{\gen}G'$ is such that
$ \whit^{\psi_{N'}^{-1}}(\varphi')=\prod_v W_v'$, then
for any sufficiently large finite set of places $S$ we have
\begin{equation}\label{unoeq: globalinner}
(\varphi',\GG(\eisen(\varphi,s)))_{G'}=\frac12\big(\prod_{j=1}^{2n+1}L(j,\chiq^j)\big)^{-1}
\frac{L^S(s+\frac12,\sigma'\otimes\pi)}{L^S(2s+1,\pi,\As^+)}
\prod_{v\in S}J_v( W_v',W_v,s).
\end{equation}
Here on the right hand side we take the unnormalized Tamagawa measures on $G'(F_S)$, $N'(F_S)$, $V(F_S)$ and $V_\gamma(F_S)$
(which are independent of the choices of gauge forms when $S$ is sufficiently large).
\end{proposition}

\begin{remark}
On p.287 of \cite{MR2848523} the formula for $\psi^{-1}_{\ell,\alpha}(\varepsilon_0^{-1}\hat z\varepsilon_0)$  (the second
display after (10.14)) is only correct in the orthogonal group case.
In the unitary group case the expression is for $\psi_{\ell,\alpha}(\varepsilon_0^{-1}\hat z\varepsilon_0)$.
This accounts for the difference between our definition of $\psi_V$ and the definition of $\psi_{\ell,\alpha}$ in \cite{MR2848523}.
\end{remark}

\subsection{Main reduction theorem}

Let $F$ be a local field.
Let $\pi\in\Irr_{\gen,\ut}\GLnn$. We say that $\pi$ is \emph{good}
if the following conditions are satisfied for all $\psi$:
\begin{enumerate}
\item $\des(\pi)$ and $\des(\c{\pi})$ are irreducible.
\item ${J}(W',W,s)$ is holomorphic at $s=\frac12$ for any $W'\in \desinv(\c{\pi})$ and $W\in\Ind(\WhitML(\pi))$.

\item For any $W'\in\desinv(\c{\pi})$,
\begin{equation} \label{eq: factorsthruGG}
J(W',W,\frac12)\text{ factors through the map }
W\mapsto(\whitformo(M(\frac12)W,\cdot,-\frac12)).
\end{equation}
\end{enumerate}
As in the case of $\une$, we can conclude that if $\pi$ is good, there exists a non-zero constant $c_\pi$ such that
\begin{equation} \label{eq: unomain}
\stint_{N'}J(\sigma'(u)W',W,\frac12)\psi_{N'}(u)\,du
=c_\pi W'(e)\whitformo(M(\frac12)W,e,-\frac12).
\end{equation}
Here $\sigma'=\desinv(\c{\pi})$.
Once again, $c_\pi$ is independent of choices of Haar measures as long as we take compatible measures on $G'$ and $U$ as in Remark~\ref{rem: indmeas}.
We note also that when $\vol(U\cap K)=1$, $\vol(K')=\big(\prod_{j=1}^{2n+1}L(j,\chiq^j)\big)^{-1}$.

As in \S\ref{sec: redeven} we can conclude:

\begin{theorem} \label{unothm: local to global}
Let $\pi\in\Ucusp\GLnn$ and $k$ as above.
Assume our Working Assumptions \ref{unoassume} and \ref{A: analyticuno}.
Then for all $v$ $\pi_v$ is good.
Moreover, let $S$ be a finite set of places including all the archimedean  places
such that $E/F$, $\pi$ and $\psi$   are unramified outside $S$. Let $\sigma=\des(\pi)$.
Then for any $ \varphi\in \sigma$ and $ \d\varphi\in\d\sigma$ which are fixed under $K'_v$   for all $v\notin S$ we have
\begin{multline}
 \whit^{\psi_{N'}}( \varphi) \whit^{\psi_{ N'}^{-1}}(\d\varphi )=
2^{1-k}(\prod_{v\in S}c_{\pi_v}^{-1})\frac{\prod_{j=1}^{2n+1}L^S(j,\chiq^j)}{L^S(1,\pi\otimes\chiqe,\As^+)}\\
(\vol(N'(\OO_S)\bs N'(F_S)))^{-1}\stint_{N'(F_S)}(\sigma(u) \varphi,\d\varphi)_{G'}\psi_{ N'}(u)^{-1}\ du.
\end{multline}
\end{theorem}

\subsection{Local conjecture}

It follows from Theorem~\ref{unothm: local to global} that Conjecture \ref{conj: global} for $\uno$  is a consequence of the following local conjecture:
\begin{conjecture} \label{unoconj: main}
Let $\pi\in \Irr_{\gen,\ut}\GLnn$ be unitarizable and good.
Then (for compatible Haar measures on $G'$ and $U$) we have $c_{\pi}=\omega'_\pi(-1)$.
\end{conjecture}

\section{A heuristic argument: case of $\U_3$}
We give a purely formal computation to support Conjecture~\ref{unoconj: main} in the case of $\U_3$.

We assume $E/F$ is inert and $n=1$. We will use the notation $\toU(x):=\sm{I_3}{x}{}{I_3}$ and $\toUbar(x):=\sm{I_3}{}{x}{I_3}$.
(Recall $G=\U_6^+$.)
For simplicity we assume that $E/F$ is unramified and fix $\rr\in E$ such that $\rr=-\c{\rr}$ and $\abs{\rr}=1$.
(In the ramified case, one has to add appropriate powers of $\abs{\rr}$ in the computations below.
The split case can be treated similarly by taking $\rr=(1,-1)$.)
We take the self-dual measures on $E$ (and $F$) with respect to $\psi$ (and $\psi(\tau\cdot)\rest_F$).

\subsection{Invariant linear form}
For any $W\in\Ind(\WhitML(\pi))$ define (see \S\ref{sec: inflate})
\begin{multline*}
L_W(g)=
\int_{F^\times}\int_{F^\times}\int_F\iint_{F^6}W_\frac12 (\levi(
\left(\begin{smallmatrix}t_1&&\\t_2 x&t_2&\\&&1\end{smallmatrix}\right))\toUbar(
\left(\begin{smallmatrix}y_1\rr&y_2\rr&y_3\rr\\y_4\rr&y_5\rr&y_2\rr\\y_6\rr&y_4\rr&y_1\rr\end{smallmatrix}\right))g )\ dy_i\ dx
\\
\abs{t_1}^{-\frac52}\abs{t_2}^{-\frac32}\ dt_2\ dt_1.
\end{multline*}
Then (once again by \cite{MR2787356}) $L_W$ is left $H$-invariant where $H={\bf r}^{-1}\GL_6(F){\bf r}\cap G\simeq\Sp_3(F)$, (${\bf r}=\diag(\rr, \rr,\rr,1,1,1)$).
Note that $L_W$ has the extra equivariance property
\[
L_W(\diag(z,z,z,\c{z}^{-1},\c{z}^{-1},\c{z}^{-1})g)=\omega_\pi(z)L_W(g).
\]
Thus as observed in \S\ref{sec: inflate}, $L_W(h\cdot)=\omega'_\pi(\lambda(h))L_W(\cdot)$ for $h\in H'$, where
\[
H'=\{h\in G:\diag(I_3,-I_3)h=\lambda(h) \c{h}\diag(I_3,-I_3)\}.
\]

Using Fourier inversion, we can express $M^*W:=(M(\frac12)W)_{-\frac12}$ in terms of $L_W$ as follows.

\begin{claim} \label{claim: M12o}
For any $W\in\Ind(\WhitML(\pi))$, 
\begin{equation}\label{eq: M12o}
M^*W(g)=\int_{(N\cap H)\bs N} L_W(\sm{}{I_3}{I_3}{}ug)\psi_N^{-1}(u)\ du.\end{equation}
\end{claim}
We note that $\sm{}{I_3}{I_3}{}\in H\levi(\rr I_3)$, thus $L_W(\sm{}{I_3}{I_3}{}\cdot)=\omega_\pi(\rr)L_W(\cdot)=\omega'_\pi(-1)L_W(\cdot)$.

\subsection{Expression for $\whitformo(M(\frac12)W,e,-\frac12)$.}

By a change of variable
$$\whitformo(W,\cdot,s)=
\int_E\int_E\int_F W_s(\toUbar(\left(\begin{smallmatrix}a &b&c\rr\\0&0&-\c{b}\\0&0&-\c{a}\end{smallmatrix}\right))\gamma\cdot )\psi^{-1}(b)\ dc\ db\ da.$$
By Fourier inversion and \eqref{eq: M12o}, we get
\begin{equation}\label{eq: whitdescn=10}
\whitformo(M(\frac12)W,e,-\frac12)=
\omega'_\pi(-1)\int_{(\alltri \cap H)\bs \alltri}\int_E L_W(u\toU(\lambda_{1,1}(1))\toUbar(\lambda_{1,1}(a))\gamma)\psi_{\alltri}^{-1}(u)\ da\ du.
\end{equation}
Here $\lambda_{1,1}(a):=a\one_{1,1}^3-\c{a}\one_{3,3}^3$; $\alltri$ is a unipotent subgroup of $G$ consisting of matrices of the form
\[
\left(\begin{smallmatrix}1&*&*&&*&*\\&1&*&&&*\\&&1&&&\\ &*&*&1&*&*\\&&*&&1&*\\&&&&&1\end{smallmatrix}\right)
\]
and $\psi_{\alltri}(u)=\psi(u_{1,2}+u_{2,3})$.

We observe that it follows from \eqref{eq: whitdescn=10} that
\begin{claim}\label{claim: central}
The descent $\des(\pi)$ has central character $\omega'_\pi$.
\end{claim}
\begin{proof}
A central element in $G'$ has the form $Z(\varepsilon)=\diag(1,\varepsilon,\sm{\frac{1+\varepsilon}{2}}{\frac{1-\varepsilon}{2}}{\frac{1-\varepsilon}{2}}{\frac{1+\varepsilon}{2}},\varepsilon,1)$ where $\varepsilon\c{\varepsilon}=1$.
It is clear that $\gamma Z(\varepsilon)\gamma^{-1}$ stabilizes the group $\{\toUbar(\lambda_{1,1}(a))\}$. Thus
\begin{multline*}
\whitformo(M(\frac12)W,Z(\varepsilon),-\frac12)=
\omega'_\pi(-1)\int_{(\alltri \cap H)\bs \alltri}\int_E\\
L_W(u\toU(\lambda_{1,1}(1))\toUbar(\lambda_{1,1}(-\frac12))(\gamma Z(\varepsilon)\gamma^{-1})\toUbar(\lambda_{1,1}(a))\gamma)\psi_{\alltri}^{-1}(u)\ da\ du.
\end{multline*}
Next let
\[{\bf a}=
{\bf a}(\varepsilon):=\toU(\lambda_{1,1}(1))\toUbar(\lambda_{1,1}(-\frac12))(\gamma Z(\varepsilon)\gamma^{-1})
\toUbar(\lambda_{1,1}(-\frac12))^{-1}\toU(\lambda_{1,1}(1))^{-1}=
\sm{\frac{1+\varepsilon}{2}I_3}{\frac{1-\varepsilon}{2}I_3}{\frac{1-\varepsilon}{2}I_3}{\frac{1+\varepsilon}{2}I_3}.
\]
Observe ${\bf a}\c{{\bf a}}=I_6$. We can check that  ${\bf a}$ stabilizes $(\alltri,\alltri\cap H,\psi_{\alltri})$. Thus we get:
\begin{multline*}\whitformo(M(\frac12)W,Z(\varepsilon),-\frac12)=\omega'_\pi(-1)\int_{(\alltri \cap H)\bs \alltri}\int_E \\
L_W({\bf a} u \toU(\lambda_{1,1}(1))\toUbar(\lambda_{1,1}(-\frac12))\toUbar(\lambda_{1,1}(a))\gamma)\psi_{\alltri}^{-1}(u)\ da\ du.
\end{multline*}
Now ${\bf a}=\diag(I_3,-I_3)\c{{\bf a}}\diag(I_3,-I_3)\varepsilon$, thus $L_W({\bf a}\cdot)=\omega'_\pi(\varepsilon)L_W(\cdot)$. A change of variable in $a$ gives the claim.
\end{proof}

\subsection{Application of a function equation}
Now we consider the left hand side of \eqref{eq: unomain}. Let $W'=\whitformod(M(\frac12)\alt{W},-\frac12)$ with $\alt{W}\in \Ind(\WhitMLd(\c{\pi}))$.
Using Bruhat decomposition, and the fact that the central character of $\desinv(c(\pi))$ is $(\omega'_\pi)^{-1}$, we get that the left hand side of \eqref{eq: unomain} is:
\[
\int_{E^{\times}}\int_{\unitE}\abs{t}^{-4}I^{\psi^{-1}}(M(\frac12)\alt{W}, t,1)I^{\psi}(W,t,z)(\omega'_\pi)^{-1}(z)\ dz\ dt
\]
where $\unitE$ is the group of norm $1$ elements in $E$ and
\[
I^{\psi}(W,t,z)=\int_{N'}\int_{V_\gamma\bs V}W(\diag(t,I_4,t^*)\gamma v w' u Z(z))\psi_V^{-1}(v)\ dv\ \psi_{N'}^{-1}(u)\ du.
\]
Here $w'$ is the Weyl element $\diag(1,w_4,1)\diag(I_2,w_2,I_2)$. We now fix a section of $V_\gamma\bs V$ to be the set of
\[
v(x,y,d):=\levi(\left(\begin{smallmatrix}1& x&y \\ &1 &0 \\ & &1\end{smallmatrix}\right))
\toU(\left(\begin{smallmatrix}0& 0&d\rr \\0 & 0&0 \\0 &0 &0\end{smallmatrix}\right)).
\]
The integral in $t$ (with the rest of the variables fixed)
has the form
\[
\int_{E^\times} W^1(\diag(t,1,1))W^2(\diag(t,1,1))\abs{t}^{-1}\ dt
\]
where $W^1\in\WhitM(\pi)$ and $W^2\in\WhitMd(\d\pi)$. We can use the function equation (split version) \cite[Theorem~1.3]{MR3220931}
and rewrite the left hand side of \eqref{eq: unomain} as:
\[
\int_{E^{\times}}\int_{\unitE}\abs{t}^{-2}I_0^{\psi^{-1}}(M(\frac12)\alt{W}, t,1)I_0^{\psi}(W,t,z)(\omega'_\pi)^{-1}(z)\ dz\ dt
\]
where
\begin{multline*}
I^{\psi}_0(W,t,z)=\int_{N'}\iiint_E\int_F W(\levi(
\left(\begin{smallmatrix}1& & \\ &1 & \\ & &t\end{smallmatrix}\right)
\left(\begin{smallmatrix}1& & \\ &1 & \\ &r&1\end{smallmatrix}\right)
\left(\begin{smallmatrix} &1& \\&&1\\1&&\end{smallmatrix}\right))
\gamma v(x,y,d) w'u Z(z))\\\psi_V^{-1}(v(x,y,d))\psi_{N'}^{-1}(u)\ dd\ dr\ dx\ dy\ du.
\end{multline*}
Using the equivariance of Whittaker functions, we may replace $I^{\psi}_0$ in the above by $I^{\psi}_b$ where $b$ is an element of the form
$\diag(1,\sm{1}{*}{}{1}\sm{1}{}{*}{1})$:
\begin{multline*}
I^{\psi}_b(W,t,z)=\int_{N'}\iiint_E\int_F W(\levi(
\left(\begin{smallmatrix}1& & \\ &1 & \\ & &t\end{smallmatrix}\right)b
\left(\begin{smallmatrix}1& & \\ &1 & \\ &r&1\end{smallmatrix}\right)
\left(\begin{smallmatrix} &1& \\&&1\\1&&\end{smallmatrix}\right))
\gamma v(x,y,d) w' u Z(z))\\\psi_V^{-1}(v(x,y,d))\psi_{N'}^{-1}(u)\ dd\ dr\ dx\ dy\ du.
\end{multline*}

We have shown:
\begin{claim}
The left hand side of \eqref{eq: unomain} equals:
$$\int_{E^{\times}}\int_{\unitE}\abs{t}^{-2}I_b^{\psi^{-1}}(M(\frac12)\alt{W}, t,1)I_b^{\psi}(W,t,z)(\omega'_\pi)^{-1}(z)\ dz\ dt$$
\end{claim}

To prove the identity $c_\pi=\omega'_\pi(-1)$, we are left to show:
\begin{claim}
When $b=\diag(1,\sm{1}{-1}{}{1})$,
\begin{equation}\label{eq: mainsplito}
  \int_{E^{\times}}\phi(t) I^{\psi}_b(M(\frac12)W,t,1)\ dt=\int_{F^{\times}}
 \abs{t}^{\frac12}\phi(t\rr)\ dt \times \whitformo(M(\frac12)W,e,-\frac12),
\end{equation}
\begin{equation}\label{eq: mainfactoro}
\int_{F^{\times}}\abs{t}^{-\frac32}\int_{\unitE} I_b^{\psi}(W,t\rr,z)(\omega'_\pi)^{-1}(z)\ dz \ dt=
\omega'_\pi(-1)\whitformo(M(\frac12)W,e,-\frac12).
\end{equation}
\end{claim}
\subsection{Proof of \eqref{eq: mainsplito}}

Write an element $u\in N'$ as
\[
\levi(\left(\begin{smallmatrix}1& & \\ &1 &s \\ & &1\end{smallmatrix}\right))
\toU(\left(\begin{smallmatrix}0& 0& 0\\ -s&c\rr & 0\\ 0&\c{s}&0\end{smallmatrix}\right)).
\]
Then explicitly
\begin{multline*}
I^{\psi}_b(W,t,z)=\iint_{E^4}\iint_{F^2} W(\levi(
\left(\begin{smallmatrix}1& & \\ &1 & \\ & &t\end{smallmatrix}\right)b
\left(\begin{smallmatrix}1& & \\ &1 & \\ &r&1\end{smallmatrix}\right))
\toUbar(\left(\begin{smallmatrix}s& -\c{x}& c\rr\\ y&d\rr & x\\ 0&-\c{y}&-\c{s}\end{smallmatrix}\right))
\\
\levi(\left(\begin{smallmatrix} 1&&\c{s} \\&1&\\&&1\end{smallmatrix}\right)
\left(\begin{smallmatrix} &1& \\&&1\\1&&\end{smallmatrix}\right))
\gamma w' Z(z))
\psi^{-1}(y)\psi^{-1}(s)\ dd\ dc\ dr\ dx\ dy\ ds.
\end{multline*}
Now let $b=\diag(1,\sm{1}{-1}{}{1})$. After a change of variables and using the equivariance of $W$, we can simplify the above as:
\begin{multline*}
I^{\psi}_b(W,t,z)=\iint_{E^4}\iint_{F^2} W(\levi(
\left(\begin{smallmatrix}1& & \\ &1 & \\ & &t\end{smallmatrix}\right)b)
\toUbar(\left(\begin{smallmatrix}s& -\c{x}& c\rr\\ y&d\rr & x\\ 0&-\c{y}&-\c{s}\end{smallmatrix}\right))
\\
\levi(\left(\begin{smallmatrix}1& & \\ &1 & \\ &r&1\end{smallmatrix}\right)
\left(\begin{smallmatrix} &1& \\&&1\\1&&\end{smallmatrix}\right))
\gamma w' Z(z))
\psi^{-1}(y)\psi^{-1}(s)\ dd\ dc\ dr\ dx\ dy\ ds.
\end{multline*}
Next conjugate $b$ and a further change of variables give:
\begin{multline}\label{eq: defIpsi}
I^{\psi}_b(W,t,z)=\iint_{E^4}\iint_{F^2} W(\levi(
\left(\begin{smallmatrix}1& & \\ &1 & \\ & &t\end{smallmatrix}\right))
\toUbar(\left(\begin{smallmatrix}s& -\c{x}& c\rr\\ y&d\rr & x\\ 0&-\c{y}&-\c{s}\end{smallmatrix}\right))
\\
\levi(b\left(\begin{smallmatrix}1& & \\ &1 & \\ &r&1\end{smallmatrix}\right)
\left(\begin{smallmatrix} &1& \\&&1\\1&&\end{smallmatrix}\right))
\gamma w' Z(z))
\psi^{-1}(s)\ dd\ dc\ dr\ dx\ dy\ ds.
\end{multline}

By Fourier inversion and \eqref{eq: M12o}, we get
\begin{multline*}
I^{\psi}_b(M(\frac12)W,t,1)=\omega'_\pi(-1)\abs{t}^3\iint_{E^3}\int_{{\Nc\cap H}\bs \Nc} L_W(u
\toUbar(\left(\begin{smallmatrix}s& 0& 0\\ y&0 & 0\\ 0&-\c{y}&-\c{s}\end{smallmatrix}\right))
\\\levi(\left(\begin{smallmatrix}1& & \\ &1 & \\ & &t\end{smallmatrix}\right)
b\left(\begin{smallmatrix}1& & \\ &1 & \\ &r&1\end{smallmatrix}\right)
\left(\begin{smallmatrix} &1& \\&&1\\1&&\end{smallmatrix}\right))\gamma w' )
\psi^{-1}(s\c{t})\psi_{\Nc}^{-1}(u)\ du\ dr\ dy\ ds.
\end{multline*}
Here $\Nc=\levi(N_\GLnn)\ltimes \{\bar u=\toUbar(\left(\begin{smallmatrix}0& *& *\\ 0&* & *\\ 0&0&0\end{smallmatrix}\right))\}$
with $\psi_{\Nc}(u\bar u)=\psi_N(u)$. Further Fourier inversion gives:
\begin{multline}\label{eq: splito}
\int_{E^{\times}}\phi(t) I^{\psi}_b(M(\frac12)W,t,1)\ dt=\omega'_\pi(-1)\int_{F^{\times}}
\abs{t}^{\frac12}\phi(t\rr)\ dt \times \\\int_{E}\int_{{\Ne\cap H}\bs \Ne}
L_W(u\levi(\left(\begin{smallmatrix}1& & \\ &1 & \\ & &\rr\end{smallmatrix}\right)
b\left(\begin{smallmatrix}1& & \\ &1 & \\ &r&1\end{smallmatrix}\right)
\left(\begin{smallmatrix} &1& \\&&1\\1&&\end{smallmatrix}\right))\gamma w' )
\psi_{\Ne}^{-1}(u)\ du\ dr.
\end{multline}
Here $\Ne=\{\levi(\left(\begin{smallmatrix}1& *& *\\0 &1& 0\\0 &0&1\end{smallmatrix}\right))\}\ltimes \bar U_\circ$ where
$\bar U_\circ=\{\toUbar(\left(\begin{smallmatrix}*& *& *\\ *&* & *\\ 0&*&*\end{smallmatrix}\right))\}$,
and $\psi_{\Ne}(u)=\psi(u_{1,2}-u_{4,1}\rr)$.

Next consider $\few=\levi(\left(\begin{smallmatrix}& & 1\\ 1&& \\ &1&\end{smallmatrix}\right))\diag(I_2,\sm{}{-\rr}{\rr^{-1}}{},I_2)$.
It is easy to see $\few\in H$, thus $L_W(\few\cdot)=L_W(\cdot)$.
Notice that $\few \Ne\few^{-1}=\alltri$ and $\psi_{\alltri}(\few u\few^{-1})=\psi_{\Ne}(u)$ for $u\in \Ne$. Let
$$\few'=\few \levi(\diag(1,1,\rr))=\levi(\left(\begin{smallmatrix}& & 1\\ 1&& \\ &1&\end{smallmatrix}\right))\diag(I_2,\sm{}{1}{1}{},I_2).$$
Then we have
\begin{equation}\label{eq: weyl}
\few' \levi(b)(\few')^{-1}=\toU(\lambda_{1,1}(1)),\few'\levi(\left(\begin{smallmatrix}1& & \\ &1 & \\ &r&1\end{smallmatrix}\right))(\few')^{-1}=
\toUbar(\lambda_{1,1}(-r)), \,\,\few'\levi(\left(\begin{smallmatrix} &1& \\&&1\\1&&\end{smallmatrix}\right))
\gamma w'=\gamma.
\end{equation}
Thus we get \eqref{eq: mainsplito} from \eqref{eq: splito} and \eqref{eq: whitdescn=10}.

\subsection{Proof of \eqref{eq: mainfactoro}}

Next we consider:
$$I':=\int_{F^{\times}}\abs{t}^{-\frac32}\int_{\unitE} I_b^{\psi}(W,t\rr,z)(\omega'_\pi)^{-1}(z)\ dz \ dt.$$
Let $$L'_W(g):=\int_{F^\times}\int_{F^\times}\int_F W_\frac12 (\levi(
\left(\begin{smallmatrix}t_1&&\\t_2 x&t_2&\\&&1\end{smallmatrix}\right))g )\ dx
\abs{t_1}^{-\frac52}\abs{t_2}^{-\frac32}\ dt_2\ dt_1.
$$
Then since $\pi$ is $\GL_3(F)$ distinguished, we have $L'_W(\levi(h)\cdot)=\abs{\det h}^2 L'_W(\cdot)$ for $h\in \GL_3(F)$.
Moreover
\begin{equation}\label{eq: leviinv}
 W(g)=\int_{(N_\GLnn\cap\GL_3(F))\bs N_\GLnn}L'_W(ug)\psi^{-1}_{N_\GLnn}(u)\ du.
\end{equation}
From this, \eqref{eq: defIpsi}  we get
\begin{multline*}
I'=\int_{F^{\times}}\abs{t}^{-\frac32}\int_{\unitE}
\iint_{E^4}\iint_{F^2} \int_{(N_\GLnn\cap\GL_3(F))\bs N_\GLnn} L'_W(u
\levi(\left(\begin{smallmatrix}1& & \\ &1 & \\ & &t\rr\end{smallmatrix}\right))
\toUbar(\left(\begin{smallmatrix}s& -\c{x}& c\rr\\ y&d\rr & x\\ 0&-\c{y}&-\c{s}\end{smallmatrix}\right))\times
\\
\levi(b\left(\begin{smallmatrix}1& & \\ &1 & \\ &r&1\end{smallmatrix}\right)
\left(\begin{smallmatrix} &1& \\&&1\\1&&\end{smallmatrix}\right))
\gamma w' Z(z))
\psi^{-1}(s)\psi^{-1}_{N_\GLnn}(u)\ du\ dd\ dc\ dr\ dx\ dy\ ds (\omega'_\pi)^{-1}(z)\ dz \ dt.
\end{multline*}
By Fourier inversion this is:
\begin{multline*}
I'=\int_{\unitE}
\iint_{E^4}\iint_{F^2}\iint_{(F\bs E)^2} L'_W(
\levi(\left(\begin{smallmatrix}1& \alpha & \beta \rr\\ &1 & \\ & &\rr\end{smallmatrix}\right))
\toUbar(\left(\begin{smallmatrix}s& -\c{x}& c\rr\\ y&d\rr & x\\ 0&-\c{y}&-\c{s}\end{smallmatrix}\right))\times
\\
\levi(b\left(\begin{smallmatrix}1& & \\ &1 & \\ &r&1\end{smallmatrix}\right)
\left(\begin{smallmatrix} &1& \\&&1\\1&&\end{smallmatrix}\right))
\gamma w' Z(z))
\psi^{-1}(s)\psi^{-1}(\alpha)\ d\alpha\ d\beta\ dd\ dc\ dr\ dx\ dy\ ds (\omega'_\pi)^{-1}(z)\ dz.
\end{multline*}

Using \eqref{eq: weyl}, we can rewrite the above as:
\begin{multline*}
I'=\int_{\unitE}
\iint_{E^4}\iint_{F^2}\iint_{(F\bs E)^2} L'_W(
\levi(\left(\begin{smallmatrix}1& \alpha & \beta \rr\\ &1 & \\ & &\rr\end{smallmatrix}\right))
\toUbar(\left(\begin{smallmatrix}s& -\c{x}& c\rr\\ y&d\rr & x\\ 0&-\c{y}&-\c{s}\end{smallmatrix}\right))\times
\\
(\few')^{-1}
\toU(\lambda_{1,1}(1))\toUbar(\lambda_{1,1}(r))\gamma  Z(z))
\psi^{-1}(s)\psi^{-1}(\alpha)\ d\alpha\ d\beta\ dd\ dc\ dr\ dx\ dy\ ds (\omega'_\pi)^{-1}(z)\ dz.
\end{multline*}
Using the element ${\bf a}(z)$ defined in the proof of ``Claim" \ref{claim: central}, this is
\begin{multline*}
I'=\int_{\unitE}
\iint_{E^4}\iint_{F^2}\iint_{(F\bs E)^2} L'_W(
\levi(\left(\begin{smallmatrix}1& \alpha & \beta \rr\\ &1 & \\ & &\rr\end{smallmatrix}\right))
\toUbar(\left(\begin{smallmatrix}s& -\c{x}& c\rr\\ y&d\rr & x\\ 0&-\c{y}&-\c{s}\end{smallmatrix}\right))\times
\\
(\few')^{-1}{\bf a}(z)
\toU(\lambda_{1,1}(1))\toUbar(\lambda_{1,1}(r))\gamma  )
\psi^{-1}(s)\psi^{-1}(\alpha)\ d\alpha\ d\beta\ dd\ dc\ dr\ dx\ dy\ ds (\omega'_\pi)^{-1}(z)\ dz.
\end{multline*}
Notice that $\levi(\left(\begin{smallmatrix}1& \alpha & \beta  \\ &1 & \\ & & 1\end{smallmatrix}\right))
\toUbar(\left(\begin{smallmatrix}s& -\c{x}& c\rr\\ y&d\rr & x\\ 0&-\c{y}&-\c{s}\end{smallmatrix}\right))
\in \Ne$ and $\Ne$ is conjugate to $\alltri$ by $\few$.
Since ${\bf a}(z)$ stabilizes $(\alltri, \alltri\cap H, \psi_{\alltri})$, we conclude that after a change of variables the above is:
\begin{multline*}
I'=\int_{\unitE}
\iint_{E^4}\iint_{F^2} \iint_{(F\bs E)^2} L'_W(\few^{-1}{\bf a}(z)\few
\levi(\left(\begin{smallmatrix}1& \alpha & \beta \\ &1 & \\ & &1\end{smallmatrix}\right))
\toUbar(\left(\begin{smallmatrix}s& -\c{x}& c\rr\\ y&d\rr & x\\ 0&-\c{y}&-\c{s}\end{smallmatrix}\right))\times
\\
\few^{-1}\toU(\lambda_{1,1}(1))\toUbar(\lambda_{1,1}(r))\gamma  )
\psi^{-1}(s)\psi^{-1}(-\rr\alpha)\ d\alpha\ d\beta\ dd\ dc\ dr\ dx\ dy\ ds (\omega'_\pi)^{-1}(z)\ dz .
\end{multline*}
The group $\{\few^{-1}{\bf a}(z)\few\}\ltimes (\Ne\cap H)$ is an open dense section of $(H'\cap P)\bs H'$.
Thus by \eqref{eq: inflate2}, the above is
\[
I'=  \int_{(\Ne\cap H)\bs \Ne} L_W(u\few^{-1}
\toU(\lambda_{1,1}(1))\toUbar(\lambda_{1,1}(r))\gamma  )
\psi_{\Ne}^{-1}(u)\ du.
\]
Since $\few\in H$ and for $u\in \Ne$,  $\few u\few^{-1}\in \alltri$ with $\psi_{\alltri}(\few u\few^{-1})=\psi_{\Ne}(u)$, we get from comparing with
\eqref{eq: whitdescn=10}
\[
I'=\omega'_\pi(-1)\whitformo(M(\frac12)W,e,-\frac12).
\]
This gives \eqref{eq: mainfactoro} and thus concludes the heuristic argument for our local conjecture $c_\pi=\omega'_\pi(-1)$ in the case of $G'=\U_3$.

\def\cprime{$'$}
\providecommand{\bysame}{\leavevmode\hbox to3em{\hrulefill}\thinspace}
\providecommand{\MR}{\relax\ifhmode\unskip\space\fi MR }
\providecommand{\MRhref}[2]{%
  \href{http://www.ams.org/mathscinet-getitem?mr=#1}{#2}
}
\providecommand{\href}[2]{#2}

\end{document}